\documentclass{amsart}
\usepackage{amssymb,latexsym,amsmath,amscd,graphicx,graphics,epic,eepic,bm,color,array,mathrsfs,fullpage}
\usepackage{enumerate}
\usepackage[all,knot,poly]{xy}

\newcommand{\nat}{\mathbb{N}}

\newcommand{\E}{\mathbb{E}}
\newcommand{\R}{\mathbb{R}}

\newcommand{\fil}{\mathcal{F}}

\newcommand{\bn}[2]{\genfrac{(}{)}{0pt}{}{#1}{#2}}

\newcommand{\ve}{\varepsilon}
\newcommand{\vf}{\varphi}

\newcommand{\id}{\mathrm{id}}

\newcommand{\Hess}{\mathrm{Hess}}
\newcommand{\M}{\mathcal{M}}
\newcommand{\g}{\gamma}
\newcommand{\adj}{\mathrm{adj}}
\newcommand{\p}{\mathcal{P}}
\newcommand{\B}{\mathcal{B}}

\theoremstyle{plain}
\newtheorem{theorem}{Theorem}[section]
\newtheorem{lemma}[theorem]{Lemma}
\newtheorem{proposition}[theorem]{Proposition}
\newtheorem{corollary}[theorem]{Corollary}

\theoremstyle{definition}
\newtheorem{definition}[theorem]{Definition}
\newtheorem{assumption}[theorem]{Assumption}

\theoremstyle{remark}

\numberwithin{theorem}{section}
\numberwithin{equation}{section}

\begin{document}

\title{Convergence of Riemannian Stochastic Gradient Descents: Varying Batch Sizes And Nonstandard Batch Forming}

\author{Hao Wu}

\thanks{The author would like to thank Conglong Xu and Peiqi Yang for interesting conversations.}

\address{Department of Mathematics, The George Washington University, Phillips Hall, Room 739, 801 22nd Street, N.W., Washington DC 20052, USA. Telephone: 1-202-994-0653, Fax: 1-202-994-6760}

\email{haowu@gwu.edu}



\begin{abstract}
We establish convergence theorems for Riemannian stochastic gradient descents in which the underlying probability spaces vary from iteration to iteration. As applications, we deduce  convergence results for Riemannian stochastic gradient descents with varying batch sizes and unbiased batch forming schemes.
\end{abstract}

\maketitle

\section{Background}\label{sec-background}

\subsection{Bonnabel's theorem} In \cite{Bonnabel}, Bonnabel established an almost sure convergence theorem for Riemannian stochastic gradient descents. Before stating this theorem, let us first recall the definition of retractions on manifolds.

\begin{definition}\cite[Definition 4.1.1]{AMS}\label{def-retraction}
Let $\M$ be a differentiable manifold. A retraction on $\M$ is a $C^2$ map $R:T\M \rightarrow \M$ such that, for every $x \in \M$, the restriction $R_x=R|_{T_x \M}$ satisfies
\begin{itemize}
	\item $R_x(\mathbf{0}_x)=x$, where $\mathbf{0}_x$ is the zero vector in $T_x \M$,
	\item $dR_x(\mathbf{0}_x)=\id_{T_x \M}$ under the standard identification $T_{\mathbf{0}_x}T_x\ M \cong T_x \M$, where $dR_x$ is the differential of $R_x$ and $\id_{T_x \M}$ is the identity map of $T_x \M$.
\end{itemize}
\end{definition}

Now we can state Bonnabel's theorem.

\begin{theorem}\label{THM-Bonnabel}\cite[Theorem 2]{Bonnabel}
Assume that
\begin{enumerate}
  \item $\M$ is a connected Riemannian manifold,\footnote{\cite[Theorem 2]{Bonnabel} assumes furthermore that the injective radius of $\M$ is uniformly bounded below by a positive number. This is unnecessary since \cite[Theorem 2]{Bonnabel} also assumes that all iterates are contained in a compact subset of $\M$.}
	\item $R:T\M \rightarrow \M$ a retraction
	\item $F:\M\rightarrow \R$ is a $C^3$ cost function,
	\item $(\Omega,\fil, \mu)$ is a probability space, where $\Omega$ is the sample space, $\fil$ is the event space and $\mu$ is the probability measure,
	\item $H:\M \times \Omega \rightarrow T\M$ is a function satisfying that
	\begin{enumerate}
		\item for every $\omega \in \Omega$, $H(\bullet, \omega): \M\rightarrow T\M$ is a tangent vector field on $\M$,
		\item for every $x \in \M$, $\E(H(x, \omega)) :=\int_\Omega H(x, \omega) d\mu = \nabla F(x)$, 
	\end{enumerate}
	\item $\{\omega_t\}_{t=0}^\infty$ is a sequence of independent random variables taking values in $\Omega$ with identical probability distribution $\mu$,
	\item $\{\g_t\}_{t=0}^\infty$ is a sequence of positive learning rates satisfying 
	\begin{equation}\label{standard-condition-learning-rates}
	\sum_{t=0}^\infty \g_t^2 < \infty \text{ and } \sum_{t=0}^\infty \g_t = \infty,
	\end{equation}
	\item $x_0 \in \M$ and $\{x_t\}_{t=0}^\infty \subset \M$ is the sequence of iterates in $\M$ generated by
	\begin{equation}\label{eq-SGD-update}
	x_{t+1} = R_{x_t}(-\g_t H(x_t,\omega_t)) \text{ for } t \geq 0, 
	\end{equation}
	\item there is a compact set $K\subset \M$ such that
	\begin{enumerate}
		\item $\{x_t\}_{t=0}^\infty \subset K$,
		\item there is an $A>0$ such that $\| H(x,\omega)\|_x \leq A$ for all $x \in K$ and $\omega \in \Omega$, where $\| \bullet\|_x$ is the Riemannian norm of $\M$ on $T_x \M$.
	\end{enumerate}
\end{enumerate}
Then $F(x_t)$ converges almost surely and $\|\nabla F(x_t)\|_{x_t} \rightarrow 0$ almost surely. 
\end{theorem}

Since $K$ is compact, the sequence $\{x_t\}$ in Theorem \ref{THM-Bonnabel} has a convergent subsequence, which almost surely converges to a stationary point of $F$.

\subsection{Batch sizes in stochastic gradient descents} 

The scope of Theorem \ref{THM-Bonnabel} is beyond basic Riemannian stochastic gradient descents. For example, for a mini-batch Riemannian stochastic gradient descent\footnote{Since all the stochastic gradient descents in the rest of this manuscript are mini-batch. We drop the word ``mini-batch" for simplicity from now on.} with fixed batch size $b>0$, that is, when updating rule \eqref{eq-SGD-update} is replaced by
\begin{equation}\label{eq-SGD-update-b}
x_{t+1} = R_{x_t}\left(-\frac{\g_t}{b} \sum_{j=tb}^{(t+1)b-1} H(x_t,\omega_j)\right) \text{ for } t \geq 0, 
\end{equation}
one can get convergence results by applying Theorem \ref{THM-Bonnabel} to the averaged random gradient 
\begin{equation}\label{eq-def-H-bar-b}
\overline{H}^{[b]}:\M \times \Omega^{b} \rightarrow T\M \text{ given by } \overline{H}^{[b]}(x,\omega_1,\dots,\omega_b) = \frac{1}{b} \sum_{j=1}^{b} H(x,\omega_j),
\end{equation}
where the probability measure on $\Omega^{b}$ is of course $\mu^{\times b}$.

More recently, there were a good amount of results about improving the efficiency of stochastic gradient descents by allowing the batch sizes to vary. See for example \cite{AHJWNOO,ASLHT,BCN,DYJG,DNG,KI,LLXLK,OZXGTK,PKKKQS,SS,UI2024,UI2025} for studies in this direction. These often involve adaptive batch sizes that increase as the cost functions decrease. Theorem \ref{THM-Bonnabel} does not directly apply to stochastic gradient descents with varying batch sizes. Many authors in this direction performed convergence analysis for their choices of adaptive batch sizes. Notably, Bottou, Curtis and Nocedal established a general mean square convergence theorem \cite[Corollary 4.12]{BCN} for stochastic gradient descents in Euclidean spaces with arbitrarily varying batch sizes.

The main benefit of using larger or increasing batch sizes is the reduction of the variance of the updating vectors, which makes stochastic gradient descents converge faster. See for example \cite{QK}. Several more sophisticated batch forming schemes were introduced to further reduce the variance of the updating vectors. Examples of such schemes can be found in \cite{LT,LX,PLW,ZKM,ZZ}.

\subsection{Contributions}

In this manuscript, we mildly generalize the bounded variance convergence argument to prove convergence theorems in the spirit of Theorem \ref{THM-Bonnabel} for Riemannian stochastic gradient descents with varying batch sizes or unbiased non-standard batch forming schemes. More precisely, we work on the following topics.
\begin{enumerate}
	\item For the bounded variance convergence argument to work, the random inputs $\{\omega_t\}_{t=0}^\infty$ need to be independent of each other. However, we observe that there is really no need for them to have the same distribution, or to even take values in the same probability space. Based on this observation, we generalize Theorem \ref{THM-Bonnabel} to Theorem \ref{THM-Omega-varying-deterministic-learning-rate} below, which allows $\omega_t$ to take values in different probability spaces as $t$ varies. Our proof of Theorem \ref{THM-Omega-varying-deterministic-learning-rate} is mostly the standard bounded variance convergence argument with some cosmetic modifications.
	\item Li and Orabona proved an almost sure convergence theorem \cite[Theorem 2]{Li-Orabona} for stochastic gradient descents in Euclidean spaces with a specific adaptive learning rate formula. This was generalized to Riemannian stochastic gradient descents with fixed batch sizes in \cite{XW-2026,YXW-2024}, where it is also observed that Li and Orabona's adaptive learning rates outperform deterministic learning rates satisfying the standard condition \eqref{standard-condition-learning-rates} in some Riemannian stochastic gradient descents. In the current manuscript, we further generalize \cite[Theorem 2]{Li-Orabona} to Theorem \ref{THM-Omega-varying-adaptive-learning-rate} below, which allows the random input $\omega_t$ to take values in different probability spaces as $t$ varies. Our proof for Theorem \ref{THM-Omega-varying-adaptive-learning-rate} closely follows the arguments by Li and Orabona in \cite{Li-Orabona}. 
	\item As long as the batches are formed independently and the expectation of the updating vector from each batch is the gradient of the cost function, one can apply Theorems \ref{THM-Omega-varying-deterministic-learning-rate} and \ref{THM-Omega-varying-adaptive-learning-rate} to get reasonably general convergence results for stochastic gradient descents. We present examples of such applications for three different batch forming schemes in Corollaries \ref{cor-batch-with-repetitions}, \ref{cor-batch-no-repetitions} and \ref{cor-batch-stratified}.
	\item Moreover, we apply the idea of ``confinements" from \cite{XYW-2024,YXW-2024} to stochastic gradient descents with varying batch sizes. This idea is rooted in \cite{Bonnabel,Bottou}. As a result, in some special and yet widely applicable cases, instead of assuming that all iterates are contained in a compact set based on some empirical evidence, we can conclusively prove this compactness assumption without running the algorithm. 
\end{enumerate}

\section{Convergence Theorems}\label{sec-THM}

In this section, we state our convergence theorems. The proofs of these theorems are contained in Section \ref{app-proofs} below.

\subsection{Retraction-dependent Lipschitz tangent vector fields}

Before stating our convergence theorems, let us introduce a notion of retraction-dependent Lipschitz tangent vector fields. Recall that if $T:V\rightarrow W$ is a linear function between two finite dimensional inner product spaces over $\R$, then there is a unique linear function $\adj(T):W \rightarrow V$, called the adjoint of $T$, satisfying that $\left\langle \mathbf{v}, \adj(T)(\mathbf{w})\right\rangle_V = \left\langle T(\mathbf{v}), \mathbf{w}\right\rangle_W$ for all $\mathbf{v} \in V$ and $\mathbf{w} \in W$, where $\left\langle \bullet, \bullet\right\rangle_V$ and $\left\langle \bullet, \bullet\right\rangle_W$ are the inner products on $V$ and $W$. If we fix orthonormal bases on $V$ and $W$, then the matrices of $T$ and $\adj(T)$ under both bases are transpositions of each other.

Let $\M$ be a Riemannian manifold, and $R:T\M \rightarrow \M$ a retraction on $\M$. For any $x \in \M$ and $\mathbf{u} \in T_x\M$, the differential $dR_x|_{\mathbf{u}}:T_{\mathbf{u}}(T_x\M) (\cong T_x\M) \rightarrow T_{R_x(\mathbf{u})} \M$ is a linear function. To simplify the notations, we will always identify $T_{\mathbf{u}}(T_x\M)$ with $T_x\M$ via the standard isomorphism. Then the adjoint of $dR_x|_{\mathbf{u}}$ is the linear function $\adj(dR_x|_{\mathbf{u}}): T_{R_x(\mathbf{u})} \M \rightarrow T_x \M$ satisfying $\left\langle \mathbf{v}, \adj(dR_x|_{\mathbf{u}})(\mathbf{w})\right\rangle_x = \left\langle dR_x|_{\mathbf{u}}(\mathbf{v}), \mathbf{w}\right\rangle_{R_x(\mathbf{u})}$ for all $\mathbf{v} \in T_x\M$ and $\mathbf{w} \in T_{R_x(\mathbf{u})} \M$, where $\left\langle \bullet, \bullet\right\rangle_x$ is the Riemannian inner product on $T_x \M$.

The following is a slight generalization of Definition \cite[Definition 2.3]{XYW-2024}. 

\begin{definition}\label{def-Lipschitz}
Let $\M$ be a Riemannian manifold, and $R:T\M \rightarrow \M$ a retraction on $\M$. For a subset $K$ of $\M$ and a positive number $r$, a vector field $\mathbf{v}$ on $\M$ is called $R$-Lipschitz on $K$ up to radius $r$ if there is a positive constant $C$ depending on $K$ and $r$ such that 
\begin{equation}\label{eq-def-Lipschitz}
\|\adj(dR_x|_{\mathbf{u}})(\mathbf{v}(R_x(\mathbf{u}))) - \mathbf{v}(x)\|_x \leq C\|\mathbf{u}\|_x
\end{equation}
for all $x \in K$ and $\mathbf{u} \in T_x \M$ satisfying $\|\mathbf{u}\|_x \leq r$.

$\mathbf{v}$ is called locally $R$-Lipschitz on $\M$ if, for every compact subset $K$ of $\M$ and every $r>0$, $\mathbf{v}$ is $R$-Lipschitz on $K$ up to radius $r$.
\end{definition}

\subsection{Two convergence theorems} Let us state some reoccurring assumptions first.

\begin{assumption}\label{assumption-cost-function}
\begin{enumerate}
  \item $\M$ is a connected Riemannian manifold.
	\item $R:T\M \rightarrow \M$ is a retraction.
	\item $F:\M\rightarrow \R$ is a $C^1$ cost function.
	\item $K$ is a compact subset of $\M$.
	\item there is a constant $A>0$ such that $\nabla F$ is $R$-Lipschitz on $K$ up to radius $A$. 
\end{enumerate}
\end{assumption}

\begin{assumption}\label{assumption-probability-space-varying}
\begin{enumerate}
	\item $\{(\Omega_t,\fil_t, \mu_t)\}_{t=0}^\infty$ is a sequence of probability spaces, and $\{H_t:\M \times \Omega_t \rightarrow T\M\}_{t=0}^\infty$ is a sequence of functions such that 
	\begin{enumerate}
		\item $H_t(\bullet, \omega): \M\rightarrow T\M$ is a tangent vector field on $\M$ for all $t\geq 0$ and $\omega \in \Omega_t$,
		\item 
		 \begin{enumerate}
		  \item either, for every $t\geq 0$, the sample space $\Omega_t$ is a finite set, and the event space $\fil_t$ consists of all the subsets of $\Omega_t$,
		  \item or, for every $t\geq 0$, $H_t:(\M \times \Omega_t, \B_{\M}\otimes \fil_t) \rightarrow (T\M, \B_{T\M})$ is measurable where
			 \begin{itemize}
				 \item $\B_{\M}$ is the Borel $\sigma$-algebra on $\M$ generated by the topology of $\M$,
				 \item $\B_{T\M}$ is the Borel $\sigma$-algebra on $T\M$ generated by the topology of $T\M$,
				 \item $\B_{\M}\otimes \fil_t$ is the product $\sigma$-algebra on $\M \times \Omega_t$ generated by $\{P\times Q~\big{|}~P\in \B_{\M}, ~Q \in \fil_t\}$.
			 \end{itemize}
	   \end{enumerate}
	  \item $\E_{\Omega_t}(H_t(x, \omega)) :=\int_{\Omega_t} H_t(x, \omega) d\mu_t = \nabla F(x)$ for all $t\geq 0$ and  $x \in \M$. 
	\end{enumerate}
	\item For the compact set $K$ and the constant $A$ in Assumption \ref{assumption-cost-function}, we have that $\| H_t(x,\omega)\|_x \leq A$ for all $t\geq 0$, $x \in K$ and $\omega \in \Omega_t$.
  \item $\{\omega_t\}_{t=0}^\infty$ is a sequence of independent random variables such that $\omega_t$ takes value in $\Omega_t$ with probability distribution $\mu_t$.
\end{enumerate}
\end{assumption}

Now we are ready to state our generalization of Bonnabel's Theorem \ref{THM-Bonnabel}.

\begin{theorem}\label{THM-Omega-varying-deterministic-learning-rate}
Suppose that Assumptions \ref{assumption-cost-function} and \ref{assumption-probability-space-varying} are true. Further assume that
\begin{enumerate}
	\item $\{\g_t\}_{t=0}^\infty$ is a sequence of positive learning rates satisfying condition \eqref{standard-condition-learning-rates} and that $\g_t\leq 1$ for $t\geq 0$,
	\item $x_0$ is a fixed point in $K$ and $\{x_t\}_{t=0}^\infty \subset \M$ is the sequence of iterates generated by 
	\begin{equation}\label{eq-SGD-update-Omega-varying-deterministic-learning-rate}
	x_{t+1} = R_{x_t}(-\g_t H_t(x_t,\omega_t)) \text{ for } t \geq 0, 
	\end{equation}
	\item $\{x_t\}_{t=0}^\infty \subset K$.
\end{enumerate}
Then $F(x_t)$ converges almost surely and $\|\nabla F(x_t)\|_{x_t} \rightarrow 0$ both almost surely and in mean square. 
\end{theorem}

Next, we generalize Li and Orabona's convergence theorem \cite[Theorem 2]{Li-Orabona}.

\begin{theorem}\label{THM-Omega-varying-adaptive-learning-rate}
Suppose that Assumptions \ref{assumption-cost-function} and \ref{assumption-probability-space-varying} are true. Further assume that
\begin{enumerate}
	\item $x_0$ is a fixed point in $K$. $\alpha,~\beta$ and $\ve$ are fixed positive numbers satisfying  $0 < \ve \leq \frac{1}{2}$ and $ 0<\alpha \leq \beta^{\frac{1}{2}+\ve}$,
	\item the sequences of iterates $\{x_t\}_{t=0}^\infty \subset \M$ and adaptive learning rates $\{\eta_t\}_{t=0}^\infty$ are  generated by
	\begin{equation}\label{eq-SGD-update-Omega-varying-adaptive-learning-rate}
	\begin{cases}
	x_{t+1} = R_{x_t}(-\eta_t H_t(x_t,\omega_t)) \\
	\eta_t = \frac{\alpha}{(\beta + \sum_{k=0}^{t-1}\|H_k(x_k,\omega_k)\|_{x_k}^2)^{\frac{1}{2}+\ve}}
	\end{cases}\text{ for } t \geq 0, 
	\end{equation}
	\item $\{x_t\}_{t=0}^\infty \subset K$.
\end{enumerate}
Then $F(x_t)$ converges almost surely and $\|\nabla F(x_t)\|_{x_t} \rightarrow 0$ almost surely. 
\end{theorem}

\subsection{Batch forming schemes} Theorems \ref{THM-Omega-varying-deterministic-learning-rate} and \ref{THM-Omega-varying-adaptive-learning-rate} apply to reasonably general stochastic gradient descents with varying batch sizes as long as the batches are formed independently and the expectation of the updating vector from each batch is the gradient of the cost function. We demonstrate this for two commonly used batch forming schemes and the stratified sampling from \cite{ZZ} in Corollaries \ref{cor-batch-with-repetitions}, \ref{cor-batch-no-repetitions} and \ref{cor-batch-stratified}. 

First, one can simply cut a sequence of independent identically distributed random variables into segments to use as batches. Note that this scheme does not preclude repetitions within each batch.

\begin{assumption}\label{assumption-probability-space-fixed}
$(\Omega,\fil, \mu)$ is a probability space, and $H:\M \times \Omega \rightarrow T\M$ is a function such that
  \begin{enumerate}
		\item $H(\bullet, \omega): \M\rightarrow T\M$ is a tangent vector field on $\M$ for every $\omega \in \Omega$,
		\item 
		 \begin{enumerate}
		  \item either the sample space $\Omega$ is a finite set, and the event space $\fil$ consists of all the subsets of $\Omega$,
		  \item or  $H:(\M \times \Omega, \B_{\M}\otimes \fil) \rightarrow (T\M, \B_{T\M})$ is measurable where
			 \begin{itemize}
				 \item $\B_{\M}$ is the Borel $\sigma$-algebra on $\M$ generated by the topology of $\M$,
				 \item $\B_{T\M}$ is the Borel $\sigma$-algebra on $T\M$ generated by the topology of $T\M$,
				 \item $\B_{\M}\otimes \fil$ is the product $\sigma$-algebra on $\M \times \Omega$ generated by $\{P\times Q~\big{|}~P\in \B_{\M}, ~Q \in \fil\}$.
			 \end{itemize}
	   \end{enumerate}
	  \item $\E(H(x, \omega)) :=\int_{\Omega} H(x, \omega) d\mu = \nabla F(x)$ for every  $x \in \M$, 
	 \item for the compact set $K$ and the constant $A$ in Assumption \ref{assumption-cost-function}, we have that $\| H(x,\omega)\|_x \leq A$ for all $x \in K$ and $\omega \in \Omega$.
	\end{enumerate}
\end{assumption}

\begin{corollary}\label{cor-batch-with-repetitions}
Suppose that Assumptions \ref{assumption-cost-function} and \ref{assumption-probability-space-fixed} are true and that 
\begin{enumerate}
  \item $\{\omega_t\}_{t=0}^\infty$ is a sequence of independent random variables such that $\omega_t$ takes value in $\Omega$ with probability distribution $\mu$,
	\item $\{S_t\}_{t=0}^\infty$ is a strictly increasing sequence of integers with $S_0 =0$.
\end{enumerate}
We have the following conclusions.
\begin{enumerate}[1.]
	\item Further assume that
\begin{enumerate}
	\item $\{\g_t\}_{t=0}^\infty$ is a sequence of positive learning rates satisfying condition \eqref{standard-condition-learning-rates} and that $\g_t\leq 1$ for $t\geq 0$,
	\item $x_0$ is a fixed point in $K$ and $\{x_t\}_{t=0}^\infty \subset \M$ is the sequence of iterates generated by
	\begin{equation}\label{eq-SGD-update-batch-size-varying-deterministic-learning-rate}
	x_{t+1} = R_{x_t}\left(-\frac{\g_t}{S_{t+1}-S_t}\sum_{j=S_t}^{S_{t+1}-1} H(x_t,\omega_j)\right) \text{ for } t \geq 0, 
	\end{equation}
	\item $\{x_t\}_{t=0}^\infty \subset K$.
\end{enumerate}
Then $F(x_t)$ converges almost surely and $\|\nabla F(x_t)\|_{x_t} \rightarrow 0$ both almost surely and in mean square.
  \item Further assume that
\begin{enumerate}
	\item $x_0$ is a fixed point in $K$. $\alpha,~\beta$ and $\ve$ are fixed positive numbers satisfying  $0 < \ve \leq \frac{1}{2}$ and $ 0<\alpha \leq \beta^{\frac{1}{2}+\ve}$,
	\item the sequences of iterates $\{x_t\}_{t=0}^\infty \subset \M$ and adaptive learning rates $\{\eta_t\}_{t=0}^\infty$ are  generated by
	\begin{equation}\label{eq-SGD-update-batch-varying-adaptive-learning-rate}
	\begin{cases}
	x_{t+1} = R_{x_t}\left(-\frac{\eta_t}{S_{t+1}-S_t}\sum_{j=S_t}^{S_{t+1}-1} H(x_t,\omega_j)\right) \\
	\eta_t = \frac{\alpha}{(\beta + \sum_{k=0}^{t-1}\|\frac{1}{S_{k+1}-S_k}\sum_{j=S_k}^{S_{k+1}-1} H(x_k,\omega_j)\|_{x_k}^2)^{\frac{1}{2}+\ve}}
	\end{cases}\text{ for } t \geq 0, 
	\end{equation}
	\item $\{x_t\}_{t=0}^\infty \subset K$.
\end{enumerate}
Then $F(x_t)$ converges almost surely and $\|\nabla F(x_t)\|_{x_t} \rightarrow 0$ almost surely. 
\end{enumerate}
\end{corollary}

It is also quite common to disallow repetitions in the same batch. To simplify the formulation, let us consider the scheme in which the batches are sampled uniformly and without repetitions in the same batch.

\begin{assumption}\label{assumption-batch-no-repetitions}
\begin{enumerate}
	\item $\Omega =\{1, 2, \dots, N\}$ is equipped with the probability measure $\mu$ given by $\mu(\{l\})=\frac{1}{N}$ for every $l \in \Omega$.
	\item For every $1\leq b\leq N$,
	\begin{enumerate}
		\item $\p_b(\Omega)$ is the set of subsets of $\Omega$ with $b$ elements,
		\item $\p_b(\Omega)$ is equipped with the probability measure $\mu_b$ given by $\mu_b(\{Y\})=\frac{1}{\bn{N}{b}}$ for every $Y \in \p_b(\Omega)$.
	\end{enumerate}
	\item $H:\M \times \Omega \rightarrow T\M$ is a function satisfying that
	\begin{enumerate}
		\item for every $l \in \Omega$, $H(\bullet, l): \M\rightarrow T\M$ is a tangent vector field on $\M$ ,
		\item for every $x \in \M$, $\E(H(x, l)) :=\frac{1}{N} \sum_{l=1}^N H(x, l) = \nabla F(x)$.
	\end{enumerate}
	\item For the compact set $K$ and the constant $A$ in Assumption \ref{assumption-cost-function}, we have that $\| H(x,l)\|_x \leq A$ for all $x \in K$ and $l \in \Omega$.
	\item $\{b_t\}_{t=0}^\infty$ is a sequence of  integers satisfying $1\leq b_t \leq N$ for $t\geq 0$.
  \item $\{Y_t\}_{t=0}^\infty$ is a sequence of independent random variables such that $Y_t$ takes value in $\p_{b_t}(\Omega)$ with probability distribution $\mu_{b_t}$.
\end{enumerate}
\end{assumption}

\begin{corollary}\label{cor-batch-no-repetitions}
Suppose that Assumptions \ref{assumption-cost-function} and \ref{assumption-batch-no-repetitions} are true. We have the following conclusions.
\begin{enumerate}[1.]
	\item Further assume that
\begin{enumerate}
	\item $\{\g_t\}_{t=0}^\infty$ is a sequence of positive learning rates satisfying condition \eqref{standard-condition-learning-rates} and that $\g_t\leq 1$ for $t\geq 0$,
	\item $x_0$ is a fixed point in $K$ and $\{x_t\}_{t=0}^\infty \subset \M$ is the sequence of iterates generated by
	\begin{equation}\label{eq-SGD-update-batch-no-repetitions-deterministic-learning-rate}
	x_{t+1} = R_{x_t}\left(-\frac{\g_t}{b_t}\sum_{y \in Y_t} H(x_t,y)\right) \text{ for } t \geq 0, 
	\end{equation}
	\item $\{x_t\}_{t=0}^\infty \subset K$.
\end{enumerate}
Then $F(x_t)$ converges almost surely and $\|\nabla F(x_t)\|_{x_t} \rightarrow 0$ both almost surely and in mean square.
  \item Further assume that
\begin{enumerate}
	\item $x_0$ is a fixed point in $K$. $\alpha,~\beta$ and $\ve$ are fixed positive numbers satisfying  $0 < \ve \leq \frac{1}{2}$ and $ 0<\alpha \leq \beta^{\frac{1}{2}+\ve}$,
	\item the sequences of iterates $\{x_t\}_{t=0}^\infty \subset \M$ and adaptive learning rates $\{\eta_t\}_{t=0}^\infty$ are  generated by
	\begin{equation}\label{eq-SGD-update-batch-no-repetitions-adaptive-learning-rate}
	\begin{cases}
	x_{t+1} = R_{x_t}\left(-\frac{\eta_t}{b_t}\sum_{y \in Y_t} H(x_t,y)\right) \\
	\eta_t = \frac{\alpha}{(\beta + \sum_{k=0}^{t-1}\|\frac{1}{b_k}\sum_{y\in Y_k} H(x_k,y)\|_{x_k}^2)^{\frac{1}{2}+\ve}}
	\end{cases}\text{ for } t \geq 0, 
	\end{equation}
	\item $\{x_t\}_{t=0}^\infty \subset K$.
\end{enumerate}
 $F(x_t)$ converges almost surely and $\|\nabla F(x_t)\|_{x_t} \rightarrow 0$ almost surely.
\end{enumerate}
\end{corollary}

Zhao and Zhang introduced a batch forming scheme called stratified sampling in \cite{ZZ}. Next, we present the applications of Theorems \ref{THM-Omega-varying-deterministic-learning-rate} and \ref{THM-Omega-varying-adaptive-learning-rate} to this scheme.

\begin{assumption}\label{assumption-batch-stratified}
Let $(\Omega,\fil, \mu)$ and $H$ be as in Assumption \ref{assumption-probability-space-fixed}. Further assume the following.
  \begin{enumerate}
  \item For every $t\geq 0$, $\{\hat{\Omega}_j^t\}_{j=1}^{m^t}$ is a partition of $\Omega$ into finitely many events with positive probabilities. That is,
	\begin{enumerate}
		\item $\Omega = \bigcup_{j=1}^{m^t} \hat{\Omega}_j^t$,
		\item $\hat{\Omega}_i^t \cap \hat{\Omega}_j^t = \emptyset$ if $i\neq j$,
		\item $\hat{\Omega}_j^t\in \fil$ and $\mu(\hat{\Omega}_j^t)>0$ for every $j=1,2,\dots,m^t$.
	\end{enumerate}
	 \item  For $t\geq 0$ and $j=1,2,\dots,m^t$, $\hat{\Omega}_j^t$ is equipped with 
	 \begin{enumerate}
		 \item the event space $\fil_j^t := \{U ~\big{|}~ U \in \fil, ~U \subset \hat{\Omega}_j^t\}$,
		 \item the probability measure $\hat{\mu}_j^t$ given by $\hat{\mu}_j^t(U):=\frac{\mu(U)}{\mu(\hat{\Omega}_j^t)}$ for $U \in \fil_j^t$.
	 \end{enumerate}
	\item For every $t\geq 0$, $(b_1^t, b_2^t,\dots,b_{m^t}^t)$ is an element of $\nat^{m^t}$, where $\nat$ is the set of positive integers.
  \item $\{\widetilde{\omega}^t\}_{t=0}^\infty$ is a sequence of independent random variables such that $\widetilde{\omega}^t$ takes value in 
	\begin{equation}\label{eq-def-wide-tilde-Omega}
	\widetilde{\Omega}^t:=(\hat{\Omega}_1^t)^{b_1^t} \times (\hat{\Omega}_2^t)^{b_2^t}\times \cdots\times (\hat{\Omega}_{m^t}^t)^{b_{m^t}^t}
	\end{equation}
	with probability distribution $(\hat{\mu}_1^t)^{\times b_1^t} \times (\hat{\mu}_2^t)^{\times b_2^t} \times \cdots\times (\hat{\mu}_{m^t}^t)^{\times b_{m^t}^t}$.
  \item For every $t\geq 0$ and $\widetilde{\omega} \in \widetilde{\Omega}^t$, write
\[
\widetilde{\omega} = ((\omega_{1,1},\dots,\omega_{1,b_1^t}),(\omega_{2,1},\dots,\omega_{2,b_2^t}),\dots,(\omega_{m^t,1},\dots,\omega_{m^t,b_{m^t}^t})),
\]
where $(\omega_{j,1},\dots,\omega_{j,b_j^t}) \in (\hat{\Omega}_j^t)^{b_j^t}$ for $j=1,\dots,m^t$. Define $\widetilde{H}^t: \M \times \widetilde{\Omega}^t \rightarrow T\M$ by 
\begin{equation}\label{eq-def-wide-tilde-H}
\widetilde{H}^t(x,\widetilde{\omega}) = \sum_{j=1}^{m^t} \frac{\mu(\hat{\Omega}_j^t)}{b_j^t} \sum_{k=1}^{b_j^t} H(x,\omega_{j,k})
\end{equation}
for $x \in \M$ and $\widetilde{\omega} \in \widetilde{\Omega}^t$. 
\end{enumerate}
\end{assumption}

\begin{corollary}\label{cor-batch-stratified}
Suppose that Assumptions \ref{assumption-cost-function}, \ref{assumption-probability-space-fixed} and \ref{assumption-batch-stratified} are true. We have the following conclusions.
\begin{enumerate}[1.]
	\item Further assume that
\begin{enumerate}
	\item $\{\g_t\}_{t=0}^\infty$ is a sequence of positive learning rates satisfying condition \eqref{standard-condition-learning-rates} and that $\g_t\leq 1$ for $t\geq 0$,
	\item $x_0$ is a fixed point in $K$ and $\{x_t\}_{t=0}^\infty \subset \M$ is the sequence of iterates generated by
	\begin{equation}\label{eq-SGD-update-batch-stratified-deterministic-learning-rate}
	x_{t+1} = R_{x_t}\left(-\g_t \widetilde{H}^t(x_t,\widetilde{\omega}^t)\right) \text{ for } t \geq 0, 
	\end{equation}
	\item $\{x_t\}_{t=0}^\infty \subset K$.
  \end{enumerate}
  Then  $F(x_t)$ converges almost surely and $\|\nabla F(x_t)\|_{x_t} \rightarrow 0$ both almost surely and in mean square.
	\item Further assume that
\begin{enumerate}
	\item $x_0$ is a fixed point in $K$. $\alpha,~\beta$ and $\ve$ are fixed positive numbers satisfying  $0 < \ve \leq \frac{1}{2}$ and $ 0<\alpha \leq \beta^{\frac{1}{2}+\ve}$,
	\item the sequences of iterates $\{x_t\}_{t=0}^\infty \subset \M$ and adaptive learning rates $\{\eta_t\}_{t=0}^\infty$ are  generated by
	\begin{equation}\label{eq-SGD-update-batch-stratified-adaptive-learning-rate}
	\begin{cases}
	x_{t+1} = R_{x_t}\left(-\eta_t \widetilde{H}^t(x_t,\widetilde{\omega}^t)\right) \\
	\eta_t = \frac{\alpha}{(\beta + \sum_{p=0}^{t-1}\| \widetilde{H}^p(x_p,\widetilde{\omega}^p)\|_{x_p}^2)^{\frac{1}{2}+\ve}}
	\end{cases}\text{ for } t \geq 0, 
	\end{equation}
	\item $\{x_t\}_{t=0}^\infty \subset K$.
  \end{enumerate}
  Then  $F(x_t)$ converges almost surely and $\|\nabla F(x_t)\|_{x_t} \rightarrow 0$ almost surely.
\end{enumerate}
\end{corollary}

\subsection{Confined stochastic gradient descents} A central assumption of Corollaries \ref{cor-batch-with-repetitions}, \ref{cor-batch-no-repetitions} and \ref{cor-batch-stratified} is that all iterates of the stochastic gradient descents are contained in a compact subset of the manifold. While this assumption can often be empirically observed after running the algorithms, it would be nice if one can theoretically predict such compactness ahead of time. For this purpose, Bottou imposed assumptions \cite[(5.2-3)]{Bottou}, and Bonnabel imposed \cite[Theorem 3, Assumptions 1-3]{Bonnabel}. Rooted in the ideas of Bottou and Bonnabel, the concept of confinements of stochastic gradient descents was introduced in \cite{XYW-2024, YXW-2024}. 

\begin{definition}\cite[Definition 2.5]{XYW-2024} \cite[Definition 2.5]{YXW-2024}\label{def-confinement}
Assume that:
\begin{enumerate}[(i)]
	\item $\M$ is a differentiable manifold with a fixed retraction $R:TM\rightarrow M$,
	\item $\Omega$ is a non-empty set,
	\item $H:\M\times \Omega \rightarrow T\M$ is a function satisfying that $H(\bullet, \omega): \M\rightarrow T\M$ is a tangent vector field on $\M$ for every $\omega \in \Omega$.
\end{enumerate}
Depending on the algorithms implemented, we have the following variants of the notion of confinements:
\begin{enumerate}[1.]
	\item A confinement of $H$ on $\M$ is a $C^2$ function $\rho: \M \rightarrow \R$ satisfying:
\begin{enumerate}
	\item For every $c\in \R$, the set $\{x \in M~\big{|}~ \rho(x)\leq c\}$ is compact.
	\item There exists a $\rho_0\in \R$ such that $\left\langle \nabla \rho (x), H(x,\omega) \right\rangle_x \geq 0$ for all $\omega \in \Omega$ and $x \in \M$ satisfying $\rho(x)\geq \rho_0$.
\end{enumerate}
  \item For a fixed $\kappa>0$, a batch $\kappa$-confinement of $H$ on $\M$ is a $C^2$ function $\rho: \M \rightarrow \R$ satisfying:
\begin{enumerate}
	\item For every $c\in \R$, the set $\{x \in M~\big{|}~ \rho(x)\leq c\}$ is compact.
	\item For every $x \in \M$, denote by $C_x$ the convex hull of $\{H(x,\omega)~\big{|}~\omega \in \Omega\}$ in $T_x \M$. There exist $\rho_0,\rho_1\in \R$ such that $\rho_0<\rho_1$ and, for all $s\in [0,\kappa]$, $x \in \M$ and $\mathbf{v} \in C_x$, 
	\begin{itemize}
		\item if $\rho(x)\leq \rho_0$, then
		\begin{equation}\label{eq-def-b-kappa-confinement-1}
    \rho\left(R_{x}\left(-s\mathbf{v}\right)\right)\leq \rho_1,
	  \end{equation}
	\item if $\rho_0\leq \rho(x)\leq \rho_1$, then
		\begin{equation}\label{eq-def-b-kappa-confinement-2}
		\left\langle \nabla \rho (x),\mathbf{v} \right\rangle_x  \geq  \max\left\{0, \frac{\kappa}{2}\Hess (\rho \circ R_x)|_{-s\mathbf{v}}\left(\mathbf{v},\mathbf{v}\right)\right\},
	  \end{equation}
		where $\Hess(\rho\circ R_x)$ is the Hessian of the function $\rho\circ R_x:T_x \M \rightarrow \R$, which is defined on the inner product space $T_x \M$.
	\end{itemize}
  \end{enumerate}
  If, in part 2(b), we only assume that inequalities \eqref{eq-def-b-kappa-confinement-1} and \eqref{eq-def-b-kappa-confinement-2} are true for $\mathbf{v} \in \{H(x,\omega)~\big{|}~\omega \in \Omega\}$, then $\rho$ is called a $\kappa$-confinement for $H$.
\end{enumerate}
\end{definition}

The assumptions \cite[(5.3)]{Bottou} and \cite[Theorem 3, Assumptions 1]{Bonnabel} are basically saying that the square of the distance function is a confinement of the gradient of the cost function. Note that assuming that the random gradient $H$ of the cost function has a confinement is more restricting than the assumption that the gradient of the cost function has a confinement. However, the benefit is that, by assuming this, we do not need the additional assumptions \cite[(5.2)]{Bottou} or \cite[Theorem 3, Assumptions 2-3]{Bonnabel}, which appear to be even more restricting. Moreover, in many regularized problems, the regularizing terms are in fact confinements of the random gradients of the regularized cost functions. For example, in the regularized low-rank approximation problems studied in \cite{XYW-2024,YXW-2024}, the regularizing terms are both confinements and $\kappa$-confinements for some $\kappa>0$ of the random gradients of the regularized cost functions. In fact, a slightly more careful analysis \cite{Yang} shows that they are also batch $\kappa$-confinements. In \cite{XW-2026}, there is also an unregularized cost function that comes with a confinement and leads to a stochastic gradient descent algorithm for linear classifications that is fundamentally different from the support vector machine. 

The following proposition shows that the stochastic gradient descent happens inside a compact set if the random gradient has a confinement. 

\begin{proposition}\label{prop-confine-compact}
Assume that
\begin{enumerate}[(i)]
	\item $\M$ is a differentiable manifold with a fixed retraction $R:TM\rightarrow M$,
	\item $\Omega$ is a non-empty set,
	\item $H:\M\times \Omega \rightarrow T\M$ is a function satisfying that 
	\begin{enumerate}
		\item $H(\bullet, \omega): \M\rightarrow T\M$ is a tangent vector field on $\M$ for every $\omega \in \Omega$,
		\item $H$ is locally bounded, that is, for any compact subset $K$ of $\M$, there is an $r>0$ depending on $K$ such that $\|H(x,\omega)\|_x\leq r$ for all $(x,\omega) \in K \times \Omega$,
	\end{enumerate}
	\item $C_x$ is the convex hull of $\{H(x,\omega)~\big{|}~\omega \in \Omega\}$ in $T_x \M$.
\end{enumerate}
Then we have the following conclusions.

\begin{enumerate}[1.]
	\item Further assume that 
	\begin{enumerate}
	  \item there is a confinement $\rho$ of $H$, and $\rho_0\in \R$ satisfies that $\left\langle \nabla \rho (x), H(x,\omega) \right\rangle_x \geq 0$ for all $\omega \in \Omega$ and $x \in \M$ satisfying $\rho(x)\geq \rho_0$,
		\item $\{\g_t\}_{t=0}^\infty$ satisfies \eqref{standard-condition-learning-rates},
		\item $\lambda,b,\Theta$ are positive constants, $c= \max \{\g_t~\big{|}~t\geq 0\}$, $\sigma=\sum_{t=0}^\infty \g_t^2$ and $\rho_1 =\rho_0+ \lambda c + \frac{b^2\sigma}{2}$,
		\item $\vf$ is a positive constant satisfying 
	\begin{equation}\label{eq-def-vf}
	\vf\geq \max\{\Lambda,B,\frac{c}{\Theta}\},
	\end{equation}
	where
	\begin{eqnarray*}
	&& \Lambda:= \frac{1}{\lambda}\sup\left\{\max\{0, ~-\left\langle  \nabla\rho(x),\mathbf{v}\right\rangle_{x}\}~\big{|}~\rho(x)\leq \rho_0,~\mathbf{v}\in C_x\right\}, \\
	&& B:= \frac{1}{b} \sup\left\{\sqrt{\max\{0, ~\Hess(\rho\circ R_{x})|_{-\theta\mathbf{v}}(\mathbf{v},\mathbf{v})\}}~\big{|}~0\leq \theta \leq \Theta,~\rho(x)\leq \rho_1,~\mathbf{v}\in C_x\right\},
	\end{eqnarray*}
		\item the sequences $\{x_t\}_{t=0}^\infty\subset \M$ and $\{\mathbf{u}_t \in C_{x_t}\}_{t=0}^\infty$ satisfy that $\rho(x_0) \leq \rho_0$ and
		\begin{equation}\label{eq-def-x-t-deterministic}
x_{t+1} = R_{x_t}\left(-\frac{\g_t}{\vf}\mathbf{u}_t\right) \text{ for } t\geq 0.
\end{equation}
	\end{enumerate}
	Then $\{x_t\}_{t=0}^\infty$ is contained in the compact set $\{x\in \M ~\big{|}~ \rho(x)\leq \rho_1\}$.
	\item Further assume that
	\begin{enumerate}
		\item there is a batch $\kappa$-confinement $\rho$ of $H$ for some $\kappa>0$, 
		\item $\rho_0<\rho_1$ satisfy \eqref{eq-def-b-kappa-confinement-1} and \eqref{eq-def-b-kappa-confinement-2}, 
		\item $\{\eta_t\}_{t=0}^\infty$ satisfies that $0<\eta_t\leq \kappa$ for $t\geq 0$,
		\item the sequences $\{x_t\}_{t=0}^\infty\subset \M$ and $\{\mathbf{u}_t \in C_{x_t}\}_{t=0}^\infty$ satisfy that $\rho(x_0) \leq \rho_1$ and
		\begin{equation}\label{eq-def-x-t-adaptive}
x_{t+1} = R_{x_t}\left(-\eta_t\mathbf{u}_t\right) \text{ for } t\geq 0.
\end{equation}
	\end{enumerate}
	Then $\{x_t\}_{t=0}^\infty$ is contained in the compact set $\{x\in \M ~\big{|}~ \rho(x)\leq \rho_1\}$.
\end{enumerate}
\end{proposition}

Combining Proposition \ref{prop-confine-compact} with Corollaries \ref{cor-batch-with-repetitions}, \ref{cor-batch-no-repetitions} and \ref{cor-batch-stratified}, one can establish the convergence of some stochastic gradient descents without explicitly assuming that all iterates are contained in a compact subset of $\M$. 

\section{Proofs}\label{app-proofs}

In this section we prove the results in Section \ref{sec-THM}. Although these proofs are fairly close to the standard bounded variance argument, we still provide most of their details for the benefits of the readers who are starting in this field. 

\subsection{The proof of Theorem \ref{THM-Omega-varying-deterministic-learning-rate}}  

The proof of Theorem \ref{THM-Omega-varying-deterministic-learning-rate} is based on that of \cite[Theorem 2.6]{XYW-2024}, which, in turn, is a close adaptation of the proof of \cite[Proposition 4]{Bertsekas-Tsitsiklis-gradient}. Unless otherwise specified, all notations in this subsection are from Theorem \ref{THM-Omega-varying-deterministic-learning-rate}.

\begin{lemma}\label{lemma-F-K-A-Lipschitz}
\begin{equation}\label{eq-nabla-F-R-adj}
\nabla (F\circ R_x)(\mathbf{v})=\mathrm{adj}(dR_x|_{\mathbf{v}})((\nabla F)(R_x(\mathbf{v}))).
\end{equation}  
Consequently, there exists a $C_1>0$ such that 
\begin{eqnarray}
\label{eq-F-K-A-Lipschitz-1} && \|\nabla (F\circ R_x)(\mathbf{v}) - \nabla F(x)\|_x \leq C_1\|\mathbf{v}\|_x, \\ 
\label{eq-F-K-A-Lipschitz-2} && F(R_x(\mathbf{v})) \leq F(x) + \left\langle \nabla F(x), \mathbf{v} \right\rangle_x + \frac{C_1}{2}\|\mathbf{v}\|_x^2
\end{eqnarray}
for every $x\in K$ and every $\mathbf{v}\in T_x M$ satisfying $\|\mathbf{v}\|_x\leq A$.
\end{lemma}

\begin{proof}
For any $x\in \M$ and $\mathbf{u},~\mathbf{v}\in T_x \M$, we have that 
\begin{eqnarray*}
&&\left\langle \nabla (F\circ R_x)(\mathbf{v}), \mathbf{u}\right\rangle_x = d(F\circ R_x)|_{\mathbf{v}}(\mathbf{u}) = ((dF|_{R_x(\mathbf{v})})\circ (dR_x|_{\mathbf{v}}))(\mathbf{u})  \\
& = & \left\langle (\nabla F)(R_x(\mathbf{v})), (dR_x|_{\mathbf{v}})(\mathbf{u})\right\rangle_{R_x(\mathbf{v})} = \left\langle \mathrm{adj}(dR_x|_{\mathbf{v}})((\nabla F)(R_x(\mathbf{v}))), \mathbf{u}\right\rangle_{x}.
\end{eqnarray*}
This proves \eqref{eq-nabla-F-R-adj}.

Since $\nabla F$ is $R$-Lipschitz on $K$ up to radius $A$, we can fix a $C_1>0$ such that
\[
\|\adj(dR_x|_{\mathbf{v}})(\nabla F(R_x(\mathbf{v}))) - \nabla F(x)\|_x \leq C_1\|\mathbf{v}\|_x
\]
for every $x \in K$ and every $\mathbf{v} \in T_x \M$ satisfying $\|\mathbf{v}\|_x\leq A$. So 
\[
\|\nabla (F\circ R_x)(\mathbf{v}) - \nabla F(x)\|_x = \|\adj(dR_x|_{\mathbf{v}})((\nabla F)(R_x(\mathbf{v}))) - \nabla F(x))\|_x  \leq C_1 \|\mathbf{v}\|_x.
\]
This proves inequality \eqref{eq-F-K-A-Lipschitz-1}. Let $p(s)=F(R_x(s\mathbf{v}))$ for $s\in [0,1]$. Then $p'(s)=\left\langle  \nabla (F\circ R_x)(s\mathbf{v}), \mathbf{v}\right\rangle_x$ and
\begin{eqnarray*}
&& F(R_x(\mathbf{v})) - F(x) = p(1) -p(0) = \int_0^1 p'(s) ds = \int_0^1 \left\langle  \nabla (F\circ R_x)(s\mathbf{v}), \mathbf{v}\right\rangle_x ds \\
& = & \int_0^1 \left\langle  \nabla F(x)+\nabla (F\circ R_x)(s\mathbf{v})-\nabla F(x), \mathbf{v}\right\rangle_x ds 
=  \left\langle  \nabla F(x), \mathbf{v}\right\rangle_x + \int_0^1 \left\langle  \nabla (F\circ R_x)(s\mathbf{v})-\nabla F(x), \mathbf{v}\right\rangle_x ds \\
& \leq & \left\langle  \nabla F(x), \mathbf{v}\right\rangle_x + \int_0^1 \|\nabla (F\circ R_x)(s\mathbf{v})-\nabla F(x)\|_x \|\mathbf{v}\|_x ds  \leq  \left\langle  \nabla F(x), \mathbf{v}\right\rangle_x + \int_0^1 C_1s \|\mathbf{v}\|_x^2 ds \\
& = & \left\langle  \nabla F(x), \mathbf{v}\right\rangle_x + \frac{C_1}{2}\|\mathbf{v}\|_x^2.
\end{eqnarray*}
This proves inequality \eqref{eq-F-K-A-Lipschitz-2}.
\end{proof}

\begin{lemma}\label{lemma-measurability-deterministic}
For every $t\geq 0$, 
\begin{enumerate}
	\item the iterate $x_t=x_t(\omega_0,\dots,\omega_{t-1})$ when viewed as a function $x_t:(\prod_{\tau =0}^{t-1} \Omega_\tau, \bigotimes_{\tau =0}^{t-1} \fil_\tau) \rightarrow (\M, \B_{\M})$ is measurable, where $\bigotimes_{\tau =0}^{t-1} \fil_t$ is the product $\sigma$-algebra generated by $\{\prod_{\tau =0}^{t-1} F_\tau ~\big{|}~ F_\tau \in\fil_\tau \text{ for } \tau=0,1,\dots, t-1\}$,
	\item the random gradient $H_t(x_t,\omega_t)$ when viewed as a function $(\prod_{\tau =0}^{t} \Omega_\tau, \bigotimes_{\tau =0}^{t} \fil_\tau) \rightarrow (T\M, \B_{T\M})$ is measurable.
\end{enumerate}
\end{lemma}

\begin{proof}
According to Assumption \ref{assumption-probability-space-varying}, Lemma \ref{lemma-measurability-deterministic} needs to be proved in two cases.

Case (i) For every $t\geq 0$, the sample space $\Omega_t$ is a finite set, and the event space $\fil_t$ consists of all the subsets of $\Omega_t$. The lemma is trivially true in this case. 

Case (ii) For every $t\geq 0$, $H_t:(\M \times \Omega_t, \B_{\M}\otimes \fil_t) \rightarrow (T\M, \B_{T\M})$ is measurable. We prove the lemma by an induction on $t$ in this case. When $t=0$, $x_0$ is a constant and therefore measurable. Furthermore, $H_0(x_0,\omega_0)$ is measurable since $H_0$ is measurable. Now assume that the lemma is true for some $t-1\geq 0$. Then $x_{t}=R(\g_{t-1}H_{t-1}(x_{t-1},\omega_{t-1})): (\prod_{\tau =0}^{t-1} \Omega_\tau, \bigotimes_{\tau =0}^{t-1} \fil_\tau) \rightarrow (\M, \B_{\M})$ is measurable since $R:(T\M,\B_{T\M}) \rightarrow (\M,\B_{\M})$ is continuous and therefore measurable. So the function $(x_t,\omega_t): (\prod_{\tau =0}^{t} \Omega_\tau, \bigotimes_{\tau =0}^{t} \fil_\tau) \rightarrow (\M\times \Omega_t, \B_{\M}\otimes \fil_t)$ is measurable. Since $H_t$ is measurable, the composition $H_t(x_t,\omega_t)$ is a measurable function $(\prod_{\tau =0}^{t} \Omega_\tau, \bigotimes_{\tau =0}^{t} \fil_\tau) \rightarrow (T\M, \B_{T\M})$. This completes the induction and proves the lemma in Case (ii).
\end{proof}

\begin{lemma}\label{lemma-expectation-converges}
$\sum_{t=0}^\infty \g_t \E(\|\nabla F(x_t)\|_{x_t}^2)$ converges and, consequently, $\sum_{t=0}^\infty \g_t \|\nabla F(x_t)\|_{x_t}^2$ converges almost surely.
\end{lemma}

\begin{proof}
Since $\| H_t(x,\omega)\|_x \leq A$ and $\g_t\leq 1$ for all $t\geq 0$, $x \in K$ and $\omega \in \Omega_t$, we have that, by Lemma \ref{lemma-F-K-A-Lipschitz},  
\begin{eqnarray}
\label{eq-expectation-converges-0} F(x_{t+1}) & = & F(R_{x_t}(-\g_t H_t(x_t,\omega_t))) \leq F(x_t) - \g_t\left\langle \nabla F(x_t), H_t(x_t,\omega_t) \right\rangle_{x_t} + \frac{C_1\g_t^2}{2}\|H_t(x_t,\omega_t)\|_x^2\\
\nonumber & \leq & F(x_t) - \g_t\left\langle \nabla F(x_t), H_t(x_t,\omega_t) \right\rangle_{x_t} + \frac{C_1A^2\g_t^2}{2}.
\end{eqnarray}
Taking expectations on both sides of this inequality, we get that
\[
\E(F(x_{t+1})) \leq \E(F(x_t)) - \g_t\E(\left\langle \nabla F(x_t), H_t(x_t,\omega_t) \right\rangle_{x_t}) + \frac{C_1\g_t^2A^2}{2}.
\]
But $\omega_t$ is independent of $x_t$, which is determined by $\omega_0,\dots, \omega_{t-1}$ and $x_0$. So
\begin{equation}\label{eq-expectation-converges-1}
\E(\left\langle \nabla F(x_t), H_t(x_t,\omega_t) \right\rangle_{x_t}) = \E(\E(\left\langle \nabla F(x_t), H_t(x_t,\omega_t) \right\rangle_{x_t}~\big{|}~x_t)) = \E(\|\nabla F (x_t)\|_{x_t}^2).
\end{equation}
Thus, 
\[
\g_t \E(\|\nabla F(x_t)\|_{x_t}^2) \leq \E(F(x_t)) - \E(F(x_{t+1})) + \frac{C_1A^2\g_t^2}{2}.
\]
Summing this for $t=0,\dots,T$, we get
\[
\sum_{t=0}^T \g_t \E(\|\nabla F(x_t)\|_{x_t}^2) \leq F(x_0) - \E(F(x_{T+1})) + \frac{C_1A^2}{2}\sum_{t=0}^T\g_t^2 \leq F(x_0) - F^\ast + \frac{C_1A^2}{2}\sum_{t=0}^\infty\g_t^2,
\]
where $F^\ast$ is the minimal value of $F$ on the compact set $K$. Since $\sum_{t=0}^\infty\g_t^2 <\infty$, this shows that $\sum_{t=0}^\infty \g_t \E(\|\nabla F(x_t)\|_{x_t}^2)$ converges. It then follows that $\sum_{t=0}^\infty \g_t \|\nabla F(x_t)\|_{x_t}^2$ is finite with probability $1$, that is, $\sum_{t=0}^\infty \g_t \|\nabla F(x_t)\|_{x_t}^2$ converges almost surely. 
\end{proof}

\begin{lemma}\label{lemma-F-x-t-converge}
Both $\sum_{t=0}^\infty \g_t\left\langle \nabla F(x_t), H_t(x_t,\omega_t)\right\rangle_{x_t}$ and $\lim_{t\rightarrow \infty} F(x_t)$ converge almost surely.
\end{lemma}

\begin{proof}
Define 
\begin{equation}\label{eq-lemma-F-x-t-converge-1}
u_t = \left\langle \nabla F(x_t), H_t(x_t,\omega_t) - \nabla F(x_t)\right\rangle_{x_t} \text{ and  } z_t = \sum_{\tau=0}^t \g_\tau u_\tau.
\end{equation} 
Since $\|\nabla F(x_t)\|_{x_t} = \|\E_{\Omega_t}(H_t(x_t,\omega))\|_{x_t} \leq \E_{\Omega_t}(\|H_t(x_t,\omega)\|_{x_t}) \leq A$, we have that 
\[
|u_t| \leq \|\nabla F(x_t)\|_{x_t} (\|\nabla F(x_t)\|_{x_t} + \|H_t(x_t,\omega_t)\|_{x_t}) \leq 2A^2 \text{ for } t\geq 0.
\]
Since $\omega_t$ is independent of $\omega_0,\dots, \omega_{t-1}$ and $x_t$ is determined by $\omega_0,\dots, \omega_{t-1}$, we have that $\E(u_t~\big{|}~\omega_0,\dots,\omega_{t-1}) = 0$ for $t \geq 0$ and therefore $\E(z_t~\big{|}~\omega_0,\dots,\omega_{t-1})=z_{t-1}$. Here, note that $z_{t-1}$ is also determined by $\omega_0,\dots,\omega_{t-1}$. This shows that $\{z_t\}_{t=0}^\infty$ is a Martingale with respect to the sequence $\{\omega_t\}_{t=0}^\infty$ of random inputs. Moreover, for $t\geq 0$, we have $\E(u_t) = \E(\E(u_t~\big{|}~\omega_0,\dots,\omega_{t-1}))=0$ and, therefore, $\E(z_t)=0$. So the variance of $z_t$ satisfies
\begin{eqnarray*}
&& Var(z_t)  =  \E(z_t^2) =E((z_{t-1}+\g_t u_t)^2) = \E(z_{t-1}^2 + \g_t^2 u_t^2 +  2\g_t u_t z_{t-1}) \\
& = &   Var(z_{t-1}) + \g_t^2\E(u_t^2) + 2 \g_t\E( u_t z_{t-1}) \leq  Var(z_{t-1}) + 4A^4\g_t^2 + 2\g_t \E(u_t z_{t-1}) \\
& = & Var(z_{t-1}) + 4A^4\g_t^2+ 2\g_t \E(z_{t-1}\E(u_t ~\big{|}~\omega_0,\dots,\omega_{t-1})) = Var(z_{t-1}) + 4A^4\g_t^2.
\end{eqnarray*}
Summing the above inequality from $1$ to $T$, we get that 
\[
Var(z_T) \leq Var(z_0) + 4A^4\sum_{t=1}^T \g_t^2 \leq Var(z_0) + 4A^4\sum_{t=1}^\infty \g_t^2.
\] 
This shows that $\{Var(z_t)\}_{t=0}^\infty$ is bounded. Thus, by the Martingale Convergence Theorem, $\lim_{t\rightarrow \infty}z_t =  \sum_{t=0}^\infty \g_t u_t$ converges almost surely. By Lemma \ref{lemma-expectation-converges}, $\sum_{t=0}^\infty \g_t \|\nabla F(x_t)\|_{x_t}^2$ also converges almost surely. Thus, $\sum_{t=0}^\infty \g_t \left\langle \nabla F(x_t), H_t(x_t,\omega_t)\right\rangle_{x_t} = \sum_{t=0}^\infty \g_t u_t + \sum_{t=0}^\infty \g_t \|\nabla F(x_t)\|_{x_t}^2$ converges almost surely.

Next consider $\{F(x_t)\}_{t=0}^\infty$. Assume that $\sum_{t=0}^\infty \g_t \left\langle \nabla F(x_t), H_t(x_t,\omega_t)\right\rangle_{x_t}$ converges, which, as we have just shown, happens with probability $1$. Define 
\[
v_t = F(x_t) -\sum_{\tau=t}^\infty \g_\tau \left\langle \nabla F(x_\tau), H_\tau (x_\tau,\omega_\tau)\right\rangle_{x_\tau} + \frac{C_1 A^2}{2} \sum_{\tau=t}^\infty \g_\tau^2.
\]
By inequality \eqref{eq-expectation-converges-0}, $\{v_t\}_{t=0}^\infty$ is a decreasing sequence. Since $\{x_t\}_{t=0}^\infty$ is contained in the compact set $K$, $\{F(x_t)\}_{t=0}^\infty$ is a bounded sequence. Since $\sum_{t=0}^\infty \g_t^2$ and $\sum_{t=0}^\infty \g_t\left\langle \nabla F(x_t), H_t(x_t,\omega_t)\right\rangle_{x_t}$ both converge, the sequences $\{\sum_{\tau=t}^\infty \g_\tau^2\}_{t=0}^\infty$ and $\{\sum_{\tau=t}^\infty \g_\tau \left\langle \nabla F(x_\tau), H_\tau (x_\tau,\omega_\tau)\right\rangle_{x_\tau}\}_{t=0}^\infty$ are also bounded. So $\{v_t\}_{t=0}^\infty$ is bounded too. Thus, $\lim_{t \rightarrow \infty} v_t$ converges. Note that 
\[
\lim_{t \rightarrow \infty} \sum_{\tau=t}^\infty \g_\tau^2 = \lim_{t \rightarrow \infty} \sum_{\tau=t}^\infty \g_\tau \left\langle \nabla F(x_\tau), H_\tau(x_\tau,\omega_\tau)\right\rangle_{x_\tau} =0.
\]
This shows that $\lim_{t\rightarrow \infty} F(x_t) = \lim_{t \rightarrow \infty} v_t$ converges when $\sum_{t=0}^\infty \g_t \left\langle \nabla F(x_t), H_t(x_t,\omega_t)\right\rangle_{x_t}$ converges, which happens with probability $1$. Hence, $\lim_{t\rightarrow \infty} F(x_t)$ also converges almost surely.
\end{proof}

\begin{lemma}\cite[Lemmas 2.15]{YXW-2024}\label{lemma-gradient-difference}
For any compact subset $\widetilde{K}$ and any $r>0$, there is a constant $C_{\widetilde{K},r}>0$ such that 
\[
\left|\|\nabla F (R_x(\mathbf{v}))\|_{R_x(\mathbf{v})} - \|\nabla(F\circ R_x)(\mathbf{v})\|_x\right| \leq C_{\widetilde{K},r} \|\mathbf{v}\|_x
\] 
for every $x \in \widetilde{K}$ and every $\mathbf{v} \in T_x \M$ satisfying $\|\mathbf{v}\|_x \leq r$.
\end{lemma}

The proof of Lemma \ref{lemma-gradient-difference} is a bit lengthy. We refer the reader to \cite[Lemmas 2.11-15]{YXW-2024} for its details. To complete the proof of Theorem \ref{THM-Omega-varying-deterministic-learning-rate}, we adopt an argument used in \cite{Alber-Iusem-Solodov,Li-Orabona,Mairal}, which is based on the following lemma.

\begin{lemma}\cite[Proposition 2]{Alber-Iusem-Solodov} \cite[Lemma A.5]{Mairal} \label{lemma-Li-Orabona-1}
Let $\{a_t\}_{t=0}^\infty$ and $\{b_t\}_{t=0}^\infty$be two sequence of non-negative numbers. Assume that 
\begin{itemize}
	\item $\sum_{t=0}^\infty a_t b_t$ converges,
	\item $\sum_{t=0}^\infty a_t$ diverges,
	\item there exists an $L\geq 0$ such that $|b_{t+1}-b_t|\leq L a_t$ for $t\geq 0$. 
\end{itemize}
Then $\lim_{t\rightarrow \infty} b_t =0$.
\end{lemma}

\begin{lemma}\label{lemma-nabla-F-x-t-converge}
$\lim_{t \rightarrow \infty} \|\nabla F(x_t)\|_{x_t} =0$ almost surely and in mean square.
\end{lemma}

\begin{proof}(Adapted from \cite{Li-Orabona}.)
Let $C_1$ and $C_2=C_{K,A}$ be the positive constants given in Lemmas \ref{lemma-F-K-A-Lipschitz} and \ref{lemma-gradient-difference}. Then 
\begin{eqnarray*}
&& \left|\|\nabla F(x_{t+1})\|_{x_{t+1}}^2 -\|\nabla F(x_{t})\|_{x_{t}}^2\right| = (\|\nabla F(x_{t+1})\|_{x_{t+1}}+\|\nabla F(x_{t})\|_{x_{t}})~\big{|}\|\nabla F(x_{t+1})\|_{x_{t+1}} -\|\nabla F(x_{t})\|_{x_{t}}\big{|} \\
& \leq & 2A\big{|}\|\nabla F(x_{t+1})\|_{x_{t+1}} -\|\nabla F(x_{t})\|_{x_{t}}\big{|}\\
& = & 2A\big{|}\|\nabla F(R_{x_t}\left(-\g_t H_t(x_t,\omega_t)\right))\|_{R_{x_t}\left(-g_t H_t(x_t,\omega_t)\right)} -\|\nabla F(x_{t})\|_{x_{t}}\big{|} \\
& \leq & 2A(\big{|}\|\nabla F(R_{x_t}\left(-\g_t H_t(x_t,\omega_t)\right))\|_{R_{x_t}\left(-\g_t H_t(x_t,\omega_t)\right)} -\|\nabla (F\circ R_{x_t})\left(-\g_t H_t(x_t,\omega_t)\right)\|_{x_{t}}\big{|} \\
&& +  \big{|} \|\nabla (F\circ R_{x_t})\left(-\g_t H_t(x_t,\omega_t)\right)\|_{x_{t}}  -\|\nabla F(x_{t})\|_{x_{t}}\big{|})
\end{eqnarray*}
By Lemma \ref{lemma-gradient-difference},
\begin{eqnarray*}
&& \big{|}\|\nabla F(R_{x_t}\left(-\g_t H_t(x_t,\omega_t)\right))\|_{R_{x_t}\left(-\g_t H_t(x_t,\omega_t)\right)} -\|\nabla (F\circ R_{x_t})\left(-\g_t H_t(x_t,\omega_t)\right)\|_{x_{t}}\big{|} \\
& \leq &  C_2 \g_t  \|H_t(x_t,\omega_t)\|_{x_{t}} \leq C_2 A \g_t.
\end{eqnarray*}
By Lemma \ref{lemma-F-K-A-Lipschitz},
\[
\big{|} \|\nabla (F\circ R_{x_t})\left(-\g_t H_t(x_t,\omega_t)\right)\|_{x_{t}} -\|\nabla F(x_{t})\|_{x_{t}}\big{|} \leq C_1\g_t  \|H_t(x_t,\omega_t)\|_{x_{t}}\leq C_1 A \g_t.
\]
Combining the above, we have that
\begin{equation}\label{eq-gradient-square-difference-bound-dterministic}
\left|\|\nabla F(x_{t+1})\|_{x_{t+1}}^2 -\|\nabla F(x_{t})\|_{x_{t}}^2\right| \leq 2 A^2 (C_1+C_2) \g_t.
\end{equation}
Taking expectations on both sides of inequality \eqref{eq-gradient-square-difference-bound-dterministic}, we get that 
\begin{equation}\label{eq-gradient-square-difference-bound-dterministic-exp}
\left|\E(\|\nabla F(x_{t+1})\|_{x_{t+1}}^2) -\E(\|\nabla F(x_{t})\|_{x_{t}}^2)\right|\leq \E\left(\left|\|\nabla F(x_{t+1})\|_{x_{t+1}}^2 -\|\nabla F(x_{t})\|_{x_{t}}^2\right|\right) \leq 2 A^2 (C_1+C_2) \g_t.
\end{equation}

In sumary, we have
\begin{itemize}
	\item $\sum_{t=0}^\infty \g_t \E(\|\nabla F(x_t)\|_{x_t}^2)$ converges and $\sum_{t=0}^\infty \g_t \|\nabla F(x_t)\|_{x_t}^2$ converges almost surely by Lemma \ref{lemma-expectation-converges},
	\item $\sum_{t=0}^\infty \g_t = \infty$ by condition \eqref{standard-condition-learning-rates}, 
	\item inequalities \eqref{eq-gradient-square-difference-bound-dterministic} and \eqref{eq-gradient-square-difference-bound-dterministic-exp}.
\end{itemize}
Thus, by Lemma \ref{lemma-Li-Orabona-1}, we know that $\lim_{t \rightarrow \infty} \|\nabla F(x_t)\|_{x_t} =0$ almost surely and in mean square.
\end{proof}

\begin{proof}[Proof of Theorem \ref{THM-Omega-varying-deterministic-learning-rate}]
It is clear that Theorem \ref{THM-Omega-varying-deterministic-learning-rate} follows from Lemmas \ref{lemma-F-x-t-converge} and \ref{lemma-nabla-F-x-t-converge}.
\end{proof}

\subsection{The proof of Theorem \ref{THM-Omega-varying-adaptive-learning-rate}} 

Except some cosmetic modifications, our proof of Theorem \ref{THM-Omega-varying-adaptive-learning-rate} is almost identical to that of \cite[Theorem 2.7]{YXW-2024}, which is a close adaptation of Li and Orabona's proof for \cite[Theorem 2]{Li-Orabona}. The current proof shares several lemmas with the proof for Theorem \ref{THM-Omega-varying-deterministic-learning-rate}. Unless otherwise specified, all notations in this subsection are from Theorem \ref{THM-Omega-varying-adaptive-learning-rate}.

\begin{lemma}\label{lemma-measurability-adaptive}
For every $t\geq 0$, 
\begin{enumerate}
	\item the adaptive learning rate $\eta_t=\eta_t(\omega_0,\dots,\omega_{t-1})$ when viewed as a function $\eta_t:(\prod_{\tau =0}^{t-1} \Omega_\tau, \bigotimes_{\tau =0}^{t-1} \fil_\tau) \rightarrow (\R, \B)$ is measurable, where $\B$ is the standard Borel $\sigma$-algebra on $\R$,
	\item the iterate $x_t=x_t(\omega_0,\dots,\omega_{t-1})$ when viewed as a function $x_t:(\prod_{\tau =0}^{t-1} \Omega_\tau, \bigotimes_{\tau =0}^{t-1} \fil_\tau) \rightarrow (\M, \B_{\M})$ is measurable, where $\bigotimes_{\tau =0}^{t-1} \fil_t$ is the product $\sigma$-algebra generated by $\{\prod_{\tau =0}^{t-1} F_\tau ~\big{|}~ F_\tau \in\fil_\tau \text{ for } \tau=0,1,\dots, t-1\}$,
	\item the random gradient $H_t(x_t,\omega_t)$ when viewed as a function $(\prod_{\tau =0}^{t} \Omega_\tau, \bigotimes_{\tau =0}^{t} \fil_\tau) \rightarrow (T\M, \B_{T\M})$ is measurable.
\end{enumerate}
\end{lemma}

\begin{proof}
According to Assumption \ref{assumption-probability-space-varying}, Lemma \ref{lemma-measurability-deterministic} needs to be proved in two cases.

Case (i) For every $t\geq 0$, the sample space $\Omega_t$ is a finite set, and the event space $\fil_t$ consists of all the subsets of $\Omega_t$. The lemma is trivially true in this case. 

Case (ii) For every $t\geq 0$, $H_t:(\M \times \Omega_t, \B_{\M}\otimes \fil_t) \rightarrow (T\M, \B_{T\M})$ is measurable. We prove the lemma by an induction on $t$ in this case. When $t=0$, $x_0$ and $\eta_0=\frac{\alpha}{\beta^{\frac{1}{2}+\ve}}$ are constant and therefore measurable. Furthermore, $H_0(x_0,\omega_0)$ is measurable since $H_0$ is measurable. Now assume that the lemma is true for $0\leq t \leq T-1$ for some $T\geq 1$. First, since $H_t(x_t,\omega_t)$ is measurable for $0\leq t \leq T-1$, we know that $\eta_T = \frac{\alpha}{(\beta + \sum_{k=0}^{T-1}\|H_k(x_k,\omega_k)\|_{x_k}^2)^{\frac{1}{2}+\ve}}$ is measurable since the Riemannian norm as a function $(T\M, \B_{T\M}) \rightarrow (\R, \B)$ is continuous and therefore measurable. Next, note that the scalar product $\eta_{T-1}H_{T-1}(x_{T-1}, \omega_{T-1})$ is a measurable function  $(\prod_{\tau =0}^{T-1} \Omega_\tau, \bigotimes_{\tau =0}^{t} \fil_\tau) \rightarrow (T\M, \B_{T\M})$. Hence, $x_{T}=R(\eta_{T-1}H_{T-1}(x_{T-1},\omega_{T-1})): (\prod_{\tau =0}^{T-1} \Omega_\tau, \bigotimes_{\tau =0}^{T-1} \fil_\tau) \rightarrow (\M, \B_{\M})$ is measurable since $R:(T\M,\B_{T\M}) \rightarrow (\M,\B_{\M})$ is continuous and therefore measurable.  Finally, note that the function $(x_T,\omega_T): (\prod_{\tau =0}^{T} \Omega_\tau, \bigotimes_{\tau =0}^{T} \fil_\tau) \rightarrow (\M\times \Omega_T, \B_{\M}\otimes \fil_T)$ is measurable. Since $H_T$ is measurable, the composition $H_T(x_T,\omega_T)$ is a measurable function $(\prod_{\tau =0}^{T} \Omega_\tau, \bigotimes_{\tau =0}^{T} \fil_\tau) \rightarrow (T\M, \B_{T\M})$. This completes the induction and proves the lemma in Case (ii).
\end{proof}

\begin{lemma}\label{lemma-bound-F-gradient-sum}
Recall that $F^{\ast}= \min\{F(x)~\big{|}~x\in K\}$. Then, for any $T\geq 1$, 
\begin{equation}\label{eq-bound-F-gradient-sum}
\E(\sum_{t=0}^T \eta_t \|\nabla F(x_t)\|_{x_t}^2) \leq F(x_0) - F^{\ast} + \frac{C_1}{2}\E(\sum_{t=0}^T \eta_t^2 \|H_t(x_t,\omega_t)\|_{x_t}^2),
\end{equation}
where $C_1$ is the positive constant from Lemma \ref{lemma-F-K-A-Lipschitz}.
\end{lemma}

\begin{proof}
Since $x_t\in K$ and $\eta_t\leq  \eta_0=\frac{\alpha}{\beta^{\frac{1}{2}+\ve}}\leq 1$ for all $t\geq 0$, we have that 
\begin{equation}\label{eq-eta-H-bound}
\| \eta_t H_t(x_t,\omega)\|_{x_t} \leq A \text{ for } t\geq 0.
\end{equation} 
Then, by Lemma \ref{lemma-F-K-A-Lipschitz},
\begin{equation}\label{eq-lemma-F-K-A-Lipschitz-ada}
F(x_{t+1}) = F(R_{x_t}(-\eta_t H_t(x_t,\omega_t)) \leq F(x_{t}) - \eta_t\left\langle \nabla F(x_t),  H_t(x_t,\omega_t)\right\rangle_{x_t} + \frac{C_1}{2} \eta_t^2\|H_t(x_t,\omega_t) \|_{x_t}^2.
\end{equation}
Taking expectations on both sides, we get that
\[
\E(F(x_{t+1})) \leq \E(F(x_{t})) - \E(\eta_t\left\langle \nabla F(x_t),  H_t(x_t,\omega_t)\right\rangle_{x_t}) + \frac{C_1}{2} \E(\eta_t^2\|H_t(x_t,\omega_t) \|_{x_t}^2)).
\]
Note that
\[
\E(\eta_t\left\langle \nabla F(x_t),  H_t(x_t,\omega_t)\right\rangle_{x_t}) = \E(\E(\eta_t\left\langle \nabla F(x_t),  H_t(x_t,\omega_t)\right\rangle_{x_t}~\big{|}~ x_t,\eta_t)) = \E(\eta_t\left\langle \nabla F(x_t),  \nabla F(x_t)\right\rangle_{x_t}).
\]
So 
\begin{equation}\label{eq-bound-F-gradient-sum-1}
\E(\eta_t \|\nabla F(x_t)\|_{x_t}^2) \leq \E(F(x_t)) - \E(F(x_{t+1})) + \frac{C_1}{2}\E( \|H_t(x_t,\omega_t)\|_{x_t}^2).
\end{equation}
Summing inequality \eqref{eq-bound-F-gradient-sum-1} from $0$ to $T$, we get that
\begin{equation}\label{eq-bound-F-gradient-sum-2}
\E(\sum_{t=0}^T \eta_t \|\nabla F(x_t)\|_{x_t}^2) \leq F(x_0) - \E(F(x_{T+1})) + \frac{C_1}{2}\E(\sum_{t=0}^T \eta_t^2 \|H_t(x_t,\omega_t)\|_{x_t}^2),
\end{equation}
where $\E(F(x_0)) = F(x_0)$ since $x_0$ is fixed. But $\E(F(x_{T+1}))\geq F^{\ast}$ since $\{x_{t}\}_{t=0}^\infty \subset K$. This prove inequality \eqref{eq-bound-F-gradient-sum}.
\end{proof}

\begin{lemma}\cite[Lemma 2]{Li-Orabona}  \label{lemma-Li-Orabona-2}
Let $a_0>1$, $a_t\geq 0$ for $t=1,\dots, T$ and $b>1$. Then 
\[
\sum_{t=1}^T \frac{a_t}{(a_0+\sum_{i=1}^ta_i)^b} \leq \frac{1}{(b-1)a_0^{b-1}}.
\]
\end{lemma}

\begin{lemma}\label{lemma-gradient-square-sum-converge}
\begin{itemize}
	\item $\sum_{t=0}^\infty \eta_t^2\| H_t(x_t, \omega_t)\|_{x_t}^2$ converges,
	\item $\sum_{t=0}^\infty \eta_t\| \nabla F(x_t)\|_{x_t}^2$ converges almost surely,
	\item $\sum_{t=0}^\infty \eta_t = \infty$.
\end{itemize}
\end{lemma}

\begin{proof}[Proof. (Following \cite{Li-Orabona}.)]
\[
\sum_{t=0}^\infty \eta_t^2\|H_t(x_t, \omega_t)\|_{x_t}^2 = \sum_{t=0}^\infty \eta_{t+1}^2\| H_t(x_t, \omega_t)\|_{x_t}^2 + \sum_{t=0}^\infty (\eta_t^2-\eta_{t+1}^2)\| H_t(x_t, \omega_t)\|_{x_t}^2.
\]
By Lemma \ref{lemma-Li-Orabona-2}, for any $T\geq 1$,
\[
\sum_{t=0}^T \eta_{t+1}^2\| H_t(x_t, \omega_t)\|_{x_t}^2 = \sum_{t=0}^T \frac{\alpha^2 \| H_t(x_t, \omega_t)\|_{x_t}^2}{(\beta+\sum_{i=0}^{t} \|H_i(x_i,\omega_i)\|_{x_i}^2)^{1+2\ve}} \leq \frac{\alpha^2}{2\ve \beta^{2\ve}}.
\]
So 
\[
\sum_{t=0}^\infty \eta_{t+1}^2\| H_t(x_t, \omega_t)\|_{x_t}^2 \leq \frac{\alpha^2}{2\ve \beta^{2\ve}}.
\]
Note that $\{\eta_t\}$ is a decreasing sequence of positive numbers. Since $\{x_{t}\}_{t=0}^\infty \subset K$, we have 
\[
\sum_{t=0}^\infty (\eta_t^2-\eta_{t+1}^2)\| H_t(x_t, \omega_t)\|_{x_t}^2 \leq \sum_{t=0}^\infty (\eta_t^2-\eta_{t+1}^2) A^2\leq A^2 \eta_0^2 = A^2 \frac{\alpha^2}{\beta^{1+2\ve}}.
\] 
Thus, $\sum_{t=0}^\infty \eta_{t}^2\| H_t(x_t, \omega_t)\|_{x_t}^2 \leq \frac{\alpha^2}{2\ve \beta^{2\ve}}+ A^2 \frac{\alpha^2}{\beta^{1+2\ve}}$ and is therefore convergent.

By Lemma \ref{lemma-bound-F-gradient-sum}, we have that 
\begin{eqnarray}
\label{eq-eta-nabla-F-x_t-sum-finite} \E(\sum_{t=0}^\infty \eta_t\| \nabla F(x_t)\|_{x_t}^2) & \leq &  F(x_0) - F^{\ast} + \frac{C_1}{2}\E(\sum_{t=0}^\infty \eta_t^2 \|H_t(x_t,\omega_t)\|_{x_t}^2) \\
\nonumber & \leq & F(x_0) - F^{\ast} + \frac{C_1}{2}(\frac{\alpha^2}{2\ve \beta^{2\ve}}+ A^2 \frac{\alpha^2}{\beta^{1+2\ve}})<\infty.
\end{eqnarray}
This implies that the probability of $\sum_{t=0}^\infty \eta_t\| \nabla F(x_t)\|_{x_t}^2<\infty$ is $1$. In other words, $\sum_{t=0}^\infty \eta_t\| \nabla F(x_t)\|_{x_t}^2$ converges almost surely.

Finally, for the series  $\sum_{t=0}^\infty \eta_t$, we have that, since $\{x_{t}\}_{t=0}^\infty \subset K$ and $\frac{1}{2}<\frac{1}{2}+\ve \leq 1$,
\[
\sum_{t=0}^\infty \eta_t = \sum_{t=0}^\infty\frac{\alpha}{(\beta+\sum_{i=0}^{t-1} \|H_t(x_i,\omega_i)\|_{x_i}^2)^{\frac{1}{2}+\ve}} \geq  \sum_{t=0}^\infty\frac{\alpha}{(\beta+t A^2)^{\frac{1}{2}+\ve}} =\infty.
\]
\end{proof}

The proof of Lemma \ref{lemma-F-x-t-converge-ada} below is very similar to that of Lemma \ref{lemma-F-x-t-converge}. 

\begin{lemma}\label{lemma-F-x-t-converge-ada}
Both $\sum_{t=0}^\infty \eta_t\left\langle \nabla F(x_t), H_t(x_t,\omega_t)\right\rangle_{x_t}$ and $\lim_{t\rightarrow \infty} F(x_t)$ converge almost surely.
\end{lemma}

\begin{proof}
Define 
\begin{equation}\label{eq-lemma-F-x-t-converge-ada-1}
u_t = \left\langle \nabla F(x_t), H_t(x_t,\omega_t) - \nabla F(x_t)\right\rangle_{x_t} \text{ and  } z_t = \sum_{\tau=0}^t \eta_\tau u_\tau.
\end{equation} 
Since $\|\nabla F(x_t)\|_{x_t} = \|\E_{\Omega_t}(H_t(x_t,\omega))\|_{x_t} \leq \E_{\Omega_t}(\|H_t(x_t,\omega)\|_{x_t}) \leq A$, we have that 
\[
|u_t| \leq \|\nabla F(x_t)\|_{x_t} (\|\nabla F(x_t)\|_{x_t} + \|H_t(x_t,\omega_t)\|_{x_t}) \leq 2A \|\nabla F(x_t)\|_{x_t} \text{ for } t\geq 0.
\]
Since $\omega_t$ is independent of $\omega_0,\dots, \omega_{t-1}$ while $x_t$ and $\eta_t$ are determined by $\omega_0,\dots, \omega_{t-1}$, we have that $\E(\eta_t u_t~\big{|}~\omega_0,\dots,\omega_{t-1}) = 0$ for $t \geq 0$ and therefore $\E(z_t~\big{|}~\omega_0,\dots,\omega_{t-1})=z_{t-1}$. Here, note that $z_{t-1}$ is also determined by $\omega_0,\dots,\omega_{t-1}$. This shows that $\{z_t\}_{t=0}^\infty$ is a Martingale with respect to the sequence $\{\omega_t\}_{t=0}^\infty$ of random inputs. Note again that $\eta_t\leq \eta_0$ for $t \geq 0$. Moreover, for $t\geq 0$, we have $\E(\eta_t u_t) = \E(\E(\eta_t u_t~\big{|}~\omega_0,\dots,\omega_{t-1}))=0$ and, therefore, $\E(z_t)=0$. So the variance of $z_t$ satisfies
\begin{eqnarray*}
&& Var(z_t)  =  \E(z_t^2) =E((z_{t-1}+\eta_t u_t)^2) = \E(z_{t-1}^2 + \eta_t^2 u_t^2 +  2\eta_t u_t z_{t-1}) \\
& = &   Var(z_{t-1}) + \E(\eta_t^2u_t^2) + 2 \E( \eta_t u_t z_{t-1}) \leq  Var(z_{t-1}) + 4A^2\eta_0\E(\eta_t \|\nabla F(x_t)\|_{x_t}^2) + 2 \E(\eta_t u_t z_{t-1}) \\
& = & Var(z_{t-1}) + 4A^2\eta_0\E(\eta_t \|\nabla F(x_t)\|_{x_t}^2)+ \E( z_{t-1}\E(\eta_t u_t~\big{|}~\omega_0,\dots,\omega_{t-1}))\\
& = & Var(z_{t-1}) + 4A^2\eta_0\E(\eta_t \|\nabla F(x_t)\|_{x_t}^2).
\end{eqnarray*}
By inequality \eqref{eq-eta-nabla-F-x_t-sum-finite}, this implies that $\{Var(z_t)\}_{t=0}^\infty$ is bounded. Thus, by the Martingale Convergence Theorem, $\lim_{t\rightarrow \infty}z_t =  \sum_{t=0}^\infty \eta_t u_t$ converges almost surely. By Lemma \ref{lemma-gradient-square-sum-converge}, $\sum_{t=0}^\infty \eta_t \|\nabla F(x_t)\|_{x_t}^2$ also converges almost surely. Thus, $\sum_{t=0}^\infty \eta_t \left\langle \nabla F(x_t), H_t(x_t,\omega_t)\right\rangle_{x_t} = \sum_{t=0}^\infty \eta_t u_t + \sum_{t=0}^\infty \eta_t \|\nabla F(x_t)\|_{x_t}^2$ converges almost surely.

Next consider $\{F(x_t)\}_{t=0}^\infty$. Assume that $\sum_{t=0}^\infty \eta_t \left\langle \nabla F(x_t), H_t(x_t,\omega_t)\right\rangle_{x_t}$ converges, which, as we have just shown, happens with probability $1$. Define 
\[
v_t = F(x_t) -\sum_{\tau=t}^\infty \eta_\tau \left\langle \nabla F(x_\tau), H_\tau (x_\tau,\omega_\tau)\right\rangle_{x_\tau} + \frac{C_1}{2} \sum_{\tau=t}^\infty \eta_\tau^2\|H_t(x_\tau,\omega_\tau) \|_{x_\tau}^2.
\]
By inequality \eqref{eq-lemma-F-K-A-Lipschitz-ada}, $\{v_t\}_{t=0}^\infty$ is a decreasing sequence. Since $\{x_t\}_{t=0}^\infty$ is contained in the compact set $K$, $\{F(x_t)\}_{t=0}^\infty$ is a bounded sequence.  Since $\sum_{t=0}^\infty \eta_t\left\langle \nabla F(x_t), H_t(x_t,\omega_t)\right\rangle_{x_t}$ converges and, by Lemma \ref{lemma-gradient-square-sum-converge}, $\sum_{t=0}^\infty \eta_t^2\| H_t(x_t, \omega_t)\|_{x_t}^2$ also converges, it follows that the sequences $\{\sum_{\tau=t}^\infty \eta_\tau \left\langle \nabla F(x_\tau), H_\tau (x_\tau,\omega_\tau)\right\rangle_{x_\tau}\}_{t=0}^\infty$ and $\{\sum_{\tau=t}^\infty \eta_\tau^2\|H_t(x_\tau,\omega_\tau) \|_{x_\tau}^2\}_{t=0}^\infty$ are also bounded. So $\{v_t\}_{t=0}^\infty$ is bounded too. Thus, $\lim_{t \rightarrow \infty} v_t$ converges. Note that 
\[
\lim_{t \rightarrow \infty} \sum_{\tau=t}^\infty \eta_\tau \left\langle \nabla F(x_\tau), H_\tau(x_\tau,\omega_\tau)\right\rangle_{x_\tau} = \lim_{t \rightarrow \infty} \sum_{\tau=t}^\infty \eta_\tau^2\|H_t(x_\tau,\omega_\tau) \|_{x_\tau}^2 =0.
\]
This shows that $\lim_{t\rightarrow \infty} F(x_t) = \lim_{t \rightarrow \infty} v_t$ converges when $\sum_{t=0}^\infty \eta_t \left\langle \nabla F(x_t), H_t(x_t,\omega_t)\right\rangle_{x_t}$ converges, which happens with probability $1$. Hence, $\lim_{t\rightarrow \infty} F(x_t)$ also converges almost surely.
\end{proof}

\begin{lemma}\label{lemma-nabla-F-x-t-converge-ada}
$\lim_{t \rightarrow \infty} \|\nabla F(x_t)\|_{x_t} =0$ almost surely.
\end{lemma}

\begin{proof} (Adapted from \cite{Li-Orabona}.)
Let $C_1$ and $C_2=C_{K,A}$ be the positive constants given in Lemmas \ref{lemma-F-K-A-Lipschitz} and \ref{lemma-gradient-difference}. Then 
\begin{eqnarray*}
&& \left|\|\nabla F(x_{t+1})\|_{x_{t+1}}^2 -\|\nabla F(x_{t})\|_{x_{t}}^2\right| = (\|\nabla F(x_{t+1})\|_{x_{t+1}}+\|\nabla F(x_{t})\|_{x_{t}})~\big{|}\|\nabla F(x_{t+1})\|_{x_{t+1}} -\|\nabla F(x_{t})\|_{x_{t}}\big{|} \\
& \leq & 2A\big{|}\|\nabla F(x_{t+1})\|_{x_{t+1}} -\|\nabla F(x_{t})\|_{x_{t}}\big{|}\\
& = & 2A\big{|}\|\nabla F(R_{x_t}\left(-\eta_t H_t(x_t,\omega_t)\right))\|_{R_{x_t}\left(-\eta_t H_t(x_t,\omega_t)\right)} -\|\nabla F(x_{t})\|_{x_{t}}\big{|} \\
& \leq & 2A(\big{|}\|\nabla F(R_{x_t}\left(-\eta_t H_t(x_t,\omega_t)\right))\|_{R_{x_t}\left(-\eta_t H_t(x_t,\omega_t)\right)} -\|\nabla (F\circ R_{x_t})\left(-\eta_t H_t(x_t,\omega_t)\right)\|_{x_{t}}\big{|} \\
&& +  \big{|} \|\nabla (F\circ R_{x_t})\left(-\eta_t H_t(x_t,\omega_t)\right)\|_{x_{t}}  -\|\nabla F(x_{t})\|_{x_{t}}\big{|})
\end{eqnarray*}
By Lemma \ref{lemma-gradient-difference} and inequality \eqref{eq-eta-H-bound},
\begin{eqnarray*}
&& \big{|}\|\nabla F(R_{x_t}\left(-\eta_t H_t(x_t,\omega_t)\right))\|_{R_{x_t}\left(-\eta_t H_t(x_t,\omega_t)\right)} -\|\nabla (F\circ R_{x_t})\left(-\eta_t H_t(x_t,\omega_t)\right)\|_{x_{t}}\big{|} \\
& \leq &  C_2 \eta_t  \|H_t(x_t,\omega_t)\|_{x_{t}} \leq C_2 A \eta_t.
\end{eqnarray*}
By Lemma \ref{lemma-F-K-A-Lipschitz} and inequality \eqref{eq-eta-H-bound},
\[
\big{|} \|\nabla (F\circ R_{x_t})\left(-\eta_t H_t(x_t,\omega_t)\right)\|_{x_{t}} -\|\nabla F(x_{t})\|_{x_{t}}\big{|} \leq C_1\eta_t  \|H_t(x_t,\omega_t)\|_{x_{t}}\leq C_1 A \eta_t.
\]
Combining the above, we have that
\begin{equation}\label{eq-gradient-square-difference-bound}
\left|\|\nabla F(x_{t+1})\|_{x_{t+1}}^2 -\|\nabla F(x_{t})\|_{x_{t}}^2\right| \leq 2 A^2 (C_1+C_2) \eta_t.
\end{equation}
In summary, we have that
\begin{itemize}
	\item $\sum_{t=0}^\infty \eta_t\| \nabla F(x_t)\|_{x_t}^2$ converges almost surely by Lemma \ref{lemma-gradient-square-sum-converge},
	\item $\sum_{t=0}^\infty \eta_t = \infty$ by Lemma \ref{lemma-gradient-square-sum-converge},
	\item $\left|\|\nabla F(x_{t+1})\|_{x_{t+1}}^2 -\|\nabla F(x_{t})\|_{x_{t}}^2\right| \leq 2 A^2 (C_1+C_2) \eta_t$ by inequality \eqref{eq-gradient-square-difference-bound}.
\end{itemize}
Thus, by Lemma \ref{lemma-Li-Orabona-1}, $\{\|\nabla F(x_t)\|_{x_t}\}_{t=0}^\infty$ converges almost surely to $0$.
\end{proof}

\begin{proof}[Proof of Theorem \ref{THM-Omega-varying-adaptive-learning-rate}]
Theorem \ref{THM-Omega-varying-adaptive-learning-rate} now follows from Lemmas \ref{lemma-F-x-t-converge-ada} and \ref{lemma-nabla-F-x-t-converge-ada}.
\end{proof}

\subsection{The proofs of Corollaries \ref{cor-batch-with-repetitions}, \ref{cor-batch-no-repetitions} and \ref{cor-batch-stratified}}

\begin{proof}[Proof of Corollary \ref{cor-batch-with-repetitions}]
Let $\overline{H}^{[b]}:\M \times \Omega^{b} \rightarrow T\M$ be the function given in \eqref{eq-def-H-bar-b}. For any $x \in \M$ and $b\geq 1$, we have that 
\[
\E_{\Omega^b}(H(x, \omega_1,\dots,\omega_b)) = \frac{1}{b}\sum_{i=1}^b\E_{\Omega}(H(x,\omega_i))= \nabla F(x).
\]
To prove Corollary \ref{cor-batch-with-repetitions}, one just needs to apply Theorems \ref{THM-Omega-varying-deterministic-learning-rate} and \ref{THM-Omega-varying-adaptive-learning-rate} to the special case in which
\begin{itemize}
	\item $\Omega_t = \Omega^{S_{t+1}-S_t}$,
	\item the random variable is $(\omega_{S_t},\dots,\omega_{S_{t+1}-1}) \in  \Omega^{S_{t+1}-S_t}$ for $t=0,1,\dots$,
	\item the random gradient is $H_t = \overline{H}^{[S_{t+1}-S_t]}: \M \times \Omega^{S_{t+1}-S_t} \rightarrow T\M$.
\end{itemize}
It is straightforward to verify that the assumptions in Corollary \ref{cor-batch-with-repetitions} imply that the above special case satisfies the corresponding assumptions in Theorems \ref{THM-Omega-varying-deterministic-learning-rate} and \ref{THM-Omega-varying-adaptive-learning-rate}.
\end{proof}

\begin{proof}[Proof of Corollary \ref{cor-batch-no-repetitions}]
Let $\overline{H}^{[b]}$ be the function defined in \eqref{eq-def-H-bar-b}. Since it is symmetric with respect to the inputs $(\omega_1,\dots,\omega_b) \in \Omega^b$,  $\overline{H}^{[b]}$ induces a function $\overline{H}^{[b]}: \M \times \p_b(\Omega) \rightarrow T\M$. Note that, for $x \in \M$ and $1 \leq b \leq N$, we have
\[
\E_{\p_b(\Omega)}(\overline{H}^{[b]}(x,Y))=\frac{1}{\bn{N}{b}} \sum_{Y \in \p_b(\Omega)} \frac{1}{b} \sum_{y\in Y} H(x,y) = \frac{\bn{N-1}{b-1}}{b\bn{N}{b}} \sum_{l=1}^N H(x,l) = \frac{1}{N} \sum_{l=1}^N H(x,l) = \nabla F(x).
\]
Now apply Theorems \ref{THM-Omega-varying-deterministic-learning-rate} and \ref{THM-Omega-varying-adaptive-learning-rate} to the special case in which
\begin{itemize}
	\item $\Omega_t = \p_{b_t}(\Omega)$,
	\item the random variable is $Y_t \in  \p_{b_t}(\Omega)$ for $t=0,1,\dots$,
	\item the random gradient is $H_t = \overline{H}^{[b_t]}: \M \times \p_{b_t}(\Omega) \rightarrow T\M$.
\end{itemize}
It is straightforward to verify that the assumptions in Corollary \ref{cor-batch-no-repetitions} imply that the above special case satisfies the corresponding assumptions in Theorems \ref{THM-Omega-varying-deterministic-learning-rate} and \ref{THM-Omega-varying-adaptive-learning-rate}.
\end{proof}

\begin{proof}[Proof of Corollary \ref{cor-batch-stratified}]
Consider the probability space $\widetilde{\Omega}^t$ given by \eqref{eq-def-wide-tilde-Omega} and the function $\widetilde{H}^t: \M \times \widetilde{\Omega}^t \rightarrow T\M$ given by \eqref{eq-def-wide-tilde-H}. Note that 
\[
\E_{\widetilde{\Omega}^t} (\widetilde{H}^t(x,\widetilde{\omega})) = \sum_{j=1}^{m^t} \frac{\mu(\hat{\Omega}_j^t)}{b_j^t} \left(b_j^t \E_{\hat{\Omega}_j^t} (H(x,\omega)) \right) = \sum_{j=1}^{m^t} \mu(\hat{\Omega}_j^t) \cdot \E_{\Omega} \left(H(x,\omega)~\big{|}~\omega \in\hat{\Omega}_j^t\right)= \E_{\Omega} (H(x,\omega)) = \nabla F(x).
\]
Now apply Theorems \ref{THM-Omega-varying-deterministic-learning-rate} and \ref{THM-Omega-varying-adaptive-learning-rate} to the special case in which
\begin{itemize}
	\item $\Omega_t = \widetilde{\Omega}^t$,
	\item the random variable is $\widetilde{\omega}^t \in \widetilde{\Omega}^t$ for $t=0,1,\dots$,
	\item the random gradient is $H_t = \widetilde{H}^t: \M \times \widetilde{\Omega}^t \rightarrow T\M$.
\end{itemize}
It is straightforward to verify that the assumptions in Corollary \ref{cor-batch-stratified} imply that the above special case satisfies the corresponding assumptions in Theorems \ref{THM-Omega-varying-deterministic-learning-rate} and \ref{THM-Omega-varying-adaptive-learning-rate}.
\end{proof}

\subsection{The proof of Proposition \ref{prop-confine-compact}} The proofs for the two conclusions in Proposition \ref{prop-confine-compact} are basically the proofs for \cite[Lemma A.2]{XYW-2024} and \cite[Lemma 2.9]{YXW-2024}.

\begin{proof}[Proof of Proposition \ref{prop-confine-compact}]
Let us prove conclusion 1 first. We prove by induction that 
\begin{equation}\label{eq-x-t-confined-induction}
\rho(x_t) + \frac{b^2}{2}\sum_{j=t}^\infty \g_j^2 \leq \rho_1,
\end{equation} 
which implies conclusion 1.

For $t=0$, $\rho(x_0)\leq \rho_0$ by our choice. So $\rho(x_0) + \frac{b^2}{2}\sum_{j=0}^\infty \g_j^2 = \rho_0 +\frac{b^2 \sigma}{2} < \rho_1$. Assume that inequality \eqref{eq-x-t-confined-induction} is true for some $t\geq 0$. Now we prove it for $t+1$. By Taylor's Theorem, there is an $s^\star \in [0,1]$ such that
\[
\rho(x_{t+1}) = \rho (R_{x_t}(-\frac{\g_t}{\vf}\mathbf{u}_t)) =\rho (x_t) - \frac{\g_t}{\vf}\left\langle  \nabla \rho(x_t), \mathbf{u}_t\right\rangle_x  + \frac{\g_t^2}{2\vf^2}\Hess(\rho\circ R_{x_t})|_{-s^\star \frac{\g_t}{\vf}\mathbf{u}_t} (\mathbf{u}_t,\mathbf{u}_t) \\
\]
Recall that $c= \max \{\g_t~\big{|}~t\geq 0\}$. Note that $0\leq s^\star \frac{\g_t}{\vf}\leq \Theta$ by inequality \eqref{eq-def-vf} and
\begin{eqnarray}
\label{eq-u-t-A}&&  - \left\langle \nabla \rho(x_t), \mathbf{u}_t\right\rangle_{x_t} \\
\nonumber& \leq &  \sup\left\{\max\{0, ~-\left\langle  \nabla\rho(x_t), \mathbf{v}\right\rangle_{x_t}\}~\big{|}~\mathbf{v} \in C_{x_t}\right\}, \\
\nonumber&& \\
\label{eq-u-t-B} && \Hess(\rho\circ R_{x_t})|_{-s^\star \frac{\g_t}{\vf}\mathbf{u}_t} (\mathbf{u}_t,\mathbf{u}_t) \\
\nonumber& \leq & \sup\left\{\max\{0, ~\Hess(\rho\circ R_{x_t})|_{-\theta\mathbf{v}}(\mathbf{v},\mathbf{v})\}~\big{|}~0\leq \theta \leq \Theta,~\mathbf{v}\in C_{x_t}\right\} \\
\nonumber &  \leq & b^2B^2.
\end{eqnarray}

Let us consider the following two cases.
\begin{itemize}
	\item[Case 1:] $\rho(x_t)\leq \rho_0$. Then, by inequalities \eqref{eq-def-vf}, \eqref{eq-u-t-A} and \eqref{eq-u-t-B}, we get 
	\[
	\rho(x_{t+1})  \leq   \rho (x_t) + \frac{\g_t}{\vf} \lambda \Lambda+ \frac{\g_t^2}{2\vf^2} b^2B^2 \leq \rho (x_t) + \lambda c + \frac{b^2\g_t^2}{2} \leq \rho_0 + \lambda c + \frac{b^2\g_t^2}{2}.
	\]
	Thus, 
	\[
	\rho(x_{t+1}) + \frac{b^2}{2}\sum_{j=t+1}^\infty \g_j^2 \leq \rho_0 + \lambda c + \frac{b^2}{2}\sum_{j=t}^\infty \g_j^2 \leq \rho_1.
	\]
	\item[Case 2:] $\rho_0<\rho(x_t) < \rho_1$. Then $\left\langle  \nabla \rho(x_t), \mathbf{u}_t\right\rangle_{x_t}\geq 0$ by our choice of $\rho_0$. So, by inequalities \eqref{eq-def-vf} and \eqref{eq-u-t-B},  
	\[
	\rho(x_{t+1}) \leq \rho (x_t)+ \frac{\g_t^2}{2\vf^2}\Hess(\rho\circ R_x)|_{-s^\star \frac{\g_t}{\vf}\mathbf{u}_t} (\mathbf{u}_t,\mathbf{u}_t) \leq  \rho (x_t)+ \frac{\g_t^2}{2\vf^2} b^2B^2 \leq \rho (x_t)+ \frac{b^2\g_t^2}{2}
	\]
	Thus, 
	\[
	\rho(x_{t+1}) + \frac{b^2}{2}\sum_{j=t+1}^\infty \g_j^2 \leq \rho(x_t) + \frac{b^2}{2}\sum_{j=t}^\infty \g_j^2 \leq \rho_1.
	\]
\end{itemize}
This proves that inequality \eqref{eq-x-t-confined-induction} is true for $t+1$ and completes the proof of conclusion 1.

Next we prove conclusion 2. The inequality $\rho(x_t) \leq \rho_1$ follows from a simpler induction. First, we know that $\rho(x_0) \leq \rho_1$ by our choice. Assuming $\rho(x_t) \leq \rho_1$ for some $t\geq 0$, let us prove that $\rho(x_{t+1}) \leq \rho_1$ too. Recall that $0< \eta_t \leq \kappa$. If $\rho(x_t) \leq \rho_0$, then it follows from inequality \eqref{eq-def-b-kappa-confinement-1} that $\rho(x_{t+1}) \leq \rho_1$. If $\rho_0<\rho(x_t) \leq \rho_1$, then by inequality \eqref{eq-def-b-kappa-confinement-2}, we have that, for some $s^\star \in [0,1]$,
\begin{eqnarray*}
	&& \rho(x_{t+1}) = \rho (R_{x_t}(-\eta_t\mathbf{u}_t)) = \rho (x_t) - \eta_t\left\langle  \nabla \rho(x_t), \mathbf{u}_t\right\rangle_{x_t}  + \frac{\eta_t^2}{2}\Hess(\rho\circ R_{x_t})|_{-s^\star \eta_t\mathbf{u}_t} (\mathbf{u}_t,\mathbf{u}_t) \\
	& = & \rho (x_t) + \eta_t \left(-\left\langle  \nabla \rho(x_t), \mathbf{u}_t\right\rangle_{x_t} + \frac{\eta_t}{2}\Hess(\rho\circ R_{x_t})|_{-s^\star \eta_t\mathbf{u}_t} (\mathbf{u}_t,\mathbf{u}_t)\right) \\
	& \leq & \rho (x_t) + \eta_t \left(-\left\langle  \nabla \rho(x_t), \mathbf{u}_t\right\rangle_{x_t} + \max\left\{ 0,\frac{\kappa}{2}\Hess(\rho\circ R_{x_t})|_{-s^\star \eta_t\mathbf{u}_t} (\mathbf{u}_t,\mathbf{u}_t)\right\}\right) \leq \rho (x_t) \leq \rho_1.
\end{eqnarray*}
This completes the induction and proves conclusion 2.
\end{proof}

\end{document}